\documentclass{amsart}
\usepackage{amssymb,amsfonts, amsmath, amsthm}
\usepackage[all,arc]{xy}
\usepackage{enumerate}
\usepackage{mathrsfs}
\usepackage{mathtools}
\usepackage[mathscr]{euscript}

\usepackage{tikz}
\usepackage{tikz-cd}
\usepackage{textcomp}
\usepackage[colorlinks=true, linkcolor=blue, citecolor = red]{hyperref}
\usetikzlibrary{patterns}
\usepackage{geometry}
\newcommand{\n}{\mathscr{N}}
\newcommand{\s}{\mathscr{S}}
\newcommand{\A}{\mathscr{A}}
\newcommand{\E}{\mathscr{E}}
\newcommand{\M}{\mathscr{M}}
\newcommand{\C}{\mathscr{C}}
\newcommand{\R}{\mathscr{R}}
\renewcommand{\P}{\mathscr{P}}
\renewcommand{\O}{\mathscr{O}}
\renewcommand{\L}{\mathscr{L}}

\newcommand{\comma}{,}
\newcommand{\set}{\mathscr{S}\text{et}}

\newcommand{\spaces}{\mathscr{S}\text{paces}}
\newcommand{\CSS}{\mathscr{C}\mathscr{S}\mathscr{S}}

\newcommand{\comsq}[8]{
  \tikzcdset{row sep/normal=0.5in}
  \tikzcdset{column sep/normal=0.5in}
  \begin{tikzcd}
    #1 \arrow[r, "#5"] \arrow[d, "#6"']
    \pgfmatrixnextcell #2 \arrow[d, "#7"] \\
    #3 \arrow[r, "#8"]
    \pgfmatrixnextcell #4
  \end{tikzcd}
}

\newcommand{\pbsq}[8]{
  \tikzcdset{row sep/normal=0.5in}
  \tikzcdset{column sep/normal=0.5in}
  \begin{tikzcd}
    #1 \arrow[r, "#5"] \arrow[d, "#6"'] \arrow[dr, phantom, "\ulcorner", very near start]
    \pgfmatrixnextcell #2 \arrow[d, "#7"] \\
    #3 \arrow[r, "#8"']
    \pgfmatrixnextcell #4
  \end{tikzcd}
}

\newcommand{\simpset}[7]{
 \begin{tikzcd}[row sep=0.5in, column sep=0.5in]
   #1 \arrow[r, shorten >=1ex,shorten <=1ex]
   \pgfmatrixnextcell #2 
   \arrow[l, shift left=1.2, "#5"] \arrow[l, shift right=1.2, "#4"'] 
   \arrow[r, shift right, shorten >=1ex,shorten <=1ex ] \arrow[r, shift left, shorten >=1ex,shorten <=1ex] 
   \pgfmatrixnextcell #3 
   \arrow[l] \arrow[l, shift left=2, "#7"] \arrow[l, shift right=2, "#6 "'] 
   \arrow[r, shorten >=1ex,shorten <=1ex] \arrow[r, shift left=2, shorten >=1ex,shorten <=1ex] \arrow[r, shift right=2, shorten >=1ex,shorten <=1ex]
   \pgfmatrixnextcell \cdots 
   \arrow[l, shift right=1] \arrow[l, shift left=1] \arrow[l, shift right=3] \arrow[l, shift left=3] 
 \end{tikzcd}
}

\newtheorem{theone}{Theorem}[section]
\newtheorem{lemone}[theone]{Lemma}
\newtheorem{propone}[theone]{Proposition}
\newtheorem{corone}[theone]{Corollary}

\theoremstyle{definition}
\newtheorem{defone}[theone]{Definition}
\newtheorem{exone}[theone]{Example}

\newtheorem{notone}[theone]{Notation}

\theoremstyle{remark}
\newtheorem{remone}[theone]{Remark}

\makeatletter
\def\@seccntformat#1{%
  \expandafter\ifx\csname c@#1\endcsname\c@section\else
  \csname the#1\endcsname\quad
  \fi}
\makeatother

\title{Complete Segal Objects}

\author{Nima Rasekh}

\date{May 2018}

\begin{document}

\begin{abstract}
We define complete Segal objects, which play the role of internal higher category objects. 
Then we study them using representable Cartesian fibrations, in particular defining adjunctions and 
limits of complete Segal objects. Finally we use Segal objects to define univalence in 
a locally Cartesian closed category that is not presentable and generalize some results 
from \cite{GK17} to the non-presentable setting.
\end{abstract}

\maketitle
\addtocontents{toc}{\protect\setcounter{tocdepth}{1}}

\tableofcontents

 \noindent

\section{Introduction} \label{Sec Introduction}

 \subsection{Motivation}\label{Subsec Motivation}
 One way to expand the tools of category theory is by introducing category objects. 
 In a given category $\C$ which has the necessary limits, we define a category object as two objects $\mathscr{O}$ 
 (the object of objects) and $\mathscr{M}$
 (the object of morphisms) with maps $(s,t): \mathscr{M} \to \mathscr{O} \times \mathscr{O}$, $id: \mathscr{O} \to \mathscr{M}$
 and $m: \mathscr{M} \times_{\mathscr{O}} \mathscr{M} \to \mathscr{M}$ that satisfy the necessary relations.
 Using this approach, we can define Lie groups as a certain category object in the category of manifolds, or
 commutative Hopf algebroids as a category object in commutative rings.
 \par 
 The goal here is to develop a similar theory in the realm of $(\infty,1)$-categories.
 Concretely, we generalize the concept of a complete Segal space that is one model of an $(\infty,1)$-category
 to internal higher category object, which we call a {\it complete Segal object}.
 \par 
 Notice this generalization is different from the one that can be found in \cite{Lu09}. 
 While both are a generalization of a complete Segal space, the work of Lurie focuses on generalizing it in a way that gives us a
 definition of an $(\infty,2)$-category, whereas here the goal is to define internal $(\infty,1)$-categories.
 In particular, a complete Segal space object in $\CSS$ is a model for an $(\infty,2)$-category, whereas
 a complete Segal object in $\CSS$ is what we would call a ``double higher category".
 See Section \ref{Sec Examples of Complete Segal Objects} for more details.
 
 \subsection{Outline}\label{Subsec Outline}
 In the first section we review some notation for simplicial spaces and some basics of complete Segal spaces.
 \par 
 In the second section we define Segal objects and complete Segal objects and develop the basic category theory.
 In particular, we discuss objects, morphisms, composition and equivalences in Segal objects.
 \par 
 In the third section we review the concept of representable Cartesian fibrations, which play the role of 
 presheaves that are represented by complete Segal objects.
 In particular, we show there is a Yoneda lemma for complete Segal objects.
 The material in this section relies on \cite{Ra17a} and \cite{Ra17b}.
 \par 
 In the fourth section we use representable Cartesian fibrations to define adjunctions and limits for complete Segal objects.
 In particular, we also prove the ``Fundamental Theorem of Complete Segal Objects".
 \par 
 In the fifth section we give some examples of complete Segal objects in various higher categories.
 In particular, talk about complete Segal objects in classical categories, spaces, right fibrations and stable higher categories.
 \par 
 In the last section we use the definition of complete Segal objects to define univalent maps in a 
 non-presentable Cartesian closed higher category, generalizing results in \cite{GK17}.
 
 \subsection{Background}
 The first two sections only need a basic understanding of complete Segal spaces.
 For the remainder we need the language of representable Cartesian fibrations which is discussed at length in \cite{Ra17b}.
 However, the main results we need are reviewed in Subsection \ref{Subsec Cartesian Fibrations} 
 and Subsection \ref{Subsec Representable Reedy Right Fibrations}. Thus as long as the reader is willing to accept the necessary results
 from \cite{Ra17b} the material here is self-contained.
 
%
 
 \subsection{Acknowledgements} \label{Subsec Acknowledgements}
 My main gratitude goes to my advisor Charles Rezk who has guided me through every single step of this work.
 I also want to thank Mike Shulman for many helpful conversations on how to relate Segal objects to univalence
 that led to the material in Section \ref{Sec Complete Segal Objects and Univalence}.
 
 
\section{Basics \& Conventions on CSS} \label{Sec Basics Conventions on CSS}
 In this section we review some of the notation and definitions we use.
 Throughout this note we use the theory of complete Segal spaces.
 The basic reference to the theory of complete Segal spaces is the original paper by Charles Rezk \cite{Re01}.
 For a discussion on adjunctions and colimits see \cite{Ra18a}. 
 Here we will only cover the basic notations. 

 Let $\s$ denote the category of simplicial sets, which we henceforth call spaces. 
 Moreover, $s\s$ is the category of bisimplicial sets, which we call simplicial spaces.
 \par
  We define $F(m)$ and $\Delta[m]$ as
  $$F(m)_{nl} = Hom_{\Delta}([n],[m]),$$
  $$\Delta[m]_{nl} = Hom_{\Delta}([l],[m]).$$
  The category $s\s$ is generated by $F(n) \times \Delta[l]$.

 We can localize the Reedy model structure on simplicial spaces to get a model for $(\infty,1)$-categories, called complete Segal spaces (CSS).
 This happens in two steps.

 \begin{defone} \label{Def Segal Spaces}
 \cite[Page 11]{Re01}
  A Reedy fibrant simplicial space $X$ is called Segal space if the map
  $$ X_n \xrightarrow{\ \ \simeq \ \ } X_1 \times_{X_0} ... \times_{X_0} X_1 $$
  is an equivalence for $n \geq 2$.
 \end{defone}
 
 \begin{theone} \label{The Segal Space Model Structure}
  \cite[Theorem 7.1]{Re01}
  There is a simplicial closed model category structure on the category $s\s$ of simplicial spaces, 
  called the Segal space model category structure, with the following properties.
  \begin{enumerate}
   \item The cofibrations are precisely the monomorphisms.
   \item The fibrant objects are precisely the Segal spaces.
   \item The weak equivalences are precisely the maps $f$ such that $Map_{s\s}(f, W)$ is
   a weak equivalence of spaces for every Segal space $W$.
   \item A Reedy weak equivalence between any two objects is a weak equivalence in
   the Segal space model category structure, and if both objects are themselves
   Segal spaces then the converse holds.
   \item The model category structure is compatible with the Cartesian closed structure
  \end{enumerate}
 \end{theone}
 A Segal space already has many characteristics of a category, such as objects and morphisms 
 \cite[Section 5]{Re01}. 
 However, it is still not an actual higher category. For that we need {\it complete Segal spaces}.
 
 \begin{defone} \label{Def Complete Segal Spaces}
  A Segal space $W$ is called a complete Segal space if it satisfies one of the following equivalent conditions
  \begin{enumerate}
   \item The map $$Map(E(1),W) \xrightarrow{\ \ \simeq \ \ } Map(F(0),W) = W_0$$
   is a trivial Kan fibration. Here $E(1)$ is the free invertible arrow.
   \item The following is a homotopy pullback square of spaces.
   \begin{center}
    \begin{tikzcd}[row sep=0.5in, column sep=0.5in]
     W_0 \arrow[r] \arrow[d] & W_3 \arrow[d] \\
     W_1 \arrow[r] & W_1 ^s \times_{W_0}^s W_1 ^t\times_{W_0}^t W_1 
    \end{tikzcd}
   \end{center}
  \end{enumerate}
 \end{defone}

 
 \begin{theone} \label{The Complete Segal Space Model Structure}
  (\cite{Re01} Theorem 7.2)
  There is a simplicial closed model category structure on the category $s\s$ of simplicial spaces, 
  called the complete Segal space model category structure, with the following properties.
  \begin{enumerate}
   \item The cofibrations are precisely the monomorphisms.
   \item The fibrant objects are precisely the complete Segal spaces.
   \item The weak equivalences are precisely the maps $f$ such that $Map_{s\s}(f, W)$ is
   a weak equivalence of spaces for every complete Segal space $W$.
   \item A Reedy weak equivalence between any two objects is a weak equivalence in
   the complete Segal space model category structure, and if both objects are themselves
   Segal spaces then the converse holds.
   \item The model category structure is compatible with the Cartesian closed structure
  \end{enumerate}
 \end{theone}
 A complete Segal space is a model for higher category and as such comes with its own category theory \cite[Section 5, 6]{Re01}.


 \section{Complete Segal Objects} \label{Sec Complete Segal Objects}
 In order to be able to define complete Segal objects, we will proceed in three steps.
 First define simplicial objects, then we specialize to Segal objects and finally we define complete Segal objects.
 
 \begin{remone} \label{Rem Context Category}
  Everything we do henceforth happens internal to a given CSS. For consistency, we will denote it by $\C$ and 
  call it the {\it ambient category}. 
  We will always assume that $\C$ has all finite limits. 
  For certain constructions we also need $\C$ to be locally Cartesian closed, which we will point out.
%
 \end{remone}

 \subsection{Simplicial Objects} \label{Simplicial Objects} 
 In this subsection we will focus on the definition of simplicial objects in $\C$. 
 
 \begin{notone} \label{Not Delta Category as CSS}
  We define the complete Segal space of simplices as the nerve $\mathcal{N}(\Delta)$ of the category of simplices $\Delta$.
  For simplicity, we will often denote this complete Segal space with $\Delta$ as well.
 \end{notone}
 
 \begin{remone}
  Concretely, $\Delta$ is a discrete simplicial space such that $\Delta_n = Fun([n],\Delta)$.
  In particular, $\Delta_0 = \mathbb{N}$. 
 \end{remone}
 
 \begin{defone} \label{Def CSS of simplicial objects}
  We define the complete Segal space of \emph{simplicial objects of $\C$} by 
  $$s \C=  \C^{(\Delta^{op})}$$ 
 \end{defone}
 
 \begin{remone}
  Given the explanation above, we can once again depict a simplicial object in $X: \Delta^{op} \to \C$ as a diagram of the form
  \begin{center}
   \simpset{X_0}{X_1}{X_2}{d_0}{d_1}{d_0}{d_1}.
  \end{center}
 \end{remone}
 
 \begin{defone} \label{Def Equiv of simp obj}
  A map of simplicial objects $ F(1) \times \Delta^{op} \to \C$ is an {\it equivalence of simplicial objects}, if for any $n \in \mathbb{N}$,
  the restriction map $F(1) \to \C$ is an equivalence in $\C$. 
  In other words, the equivalences are just level-wise equivalences of simplicial objects.
 \end{defone}

 \begin{notone}
  There are $3$ important maps in the category $\Delta$ that we will need later on and thus deserve their own names.
  \begin{enumerate}
   \item There is a map $d_1: [0] \to [1]$ taking the unique point to $0 \in [1]$. This induces a simplicial map 
   $$d_1: X_1 \to X_0$$
   We will call this map the ``source map" and denote it by $s$. 
   \item Similarly, the map $d_0: [0] \to [1]$, taking the point to $1 \in [1]$, gives us a map 
   $$d_0:X_1 \to X_0$$
   which we refer to as the ``target map" and denote by $t$.
   \item There is a unique map $s_0: [1] \to [0]$, which gives a map:
   $$s_0: X_0 \to X_1$$
  \end{enumerate}
 \end{notone}


 \subsection{Segal Objects} \label{Subsec Segal Objects}
 In this subsection we specialize our simplicial objects so that it has some categorical properties.
 Namely, we impose the {\it Segal condition}.
 Using the Segal condition we can develop many interesting categorical notions internal to $\C$.

 

 In order to be able to define Segal objects we first need some preliminary definitions. 
 
 \begin{defone} \label{Def Gn}
  Let $g(n)$ be the category with objects maps $f: [1] \to [n]$ such that $f(1) - f(0) \leq 1$,
  excluding the two constant maps with value $0$ and $n$.
  Every such map is completely determined by its images.
  Thus, we can characterize a map as a string of two integers $ij$ such that $0 \leq i,j, \leq n$ and  $j-i \leq 1$.
  In order to simplify notation, we will depict the constant map with target $i$ with in one letter $i$, rather than $ii$.
  Thus the category $g(n)$ has $2n-1$ objects that can be summarized as the set 
  $$ \{ 01,1,12,2,...,n-1,n-1n \}$$
  
  For each object $ij$ in $g(n)$ there is exactly one non-trivial morphism $0^*:ij \to i$ (if $i > 0$) 
  and exactly one non-trivial morphism $1^*:ij \to j$ (if $j <n$). 
  In particular, there are no two composable non-trivial morphisms.
  
 \end{defone}
 
 \begin{defone} \label{Def Cone Gn}
  We define the {\it cone of g(n)}, $cg(n)$, as the category which has objects of $g(n)$, plus one additional object, which we denote by 
  $01...n$. Moreover, there is one unique map from $01...n$  to each object in $g(n)$, which, depending on the target, 
  we denote by $a_{ij}$ or $a_{i}$. Note the category $cg(n)$ does have composable arrows and so we have $0^*a_{ij} = a_i$ and 
  $1^*a_{ij} = a_j$.
 \end{defone}
 
 \begin{remone}
  We can depict the category as the following diagram 

   
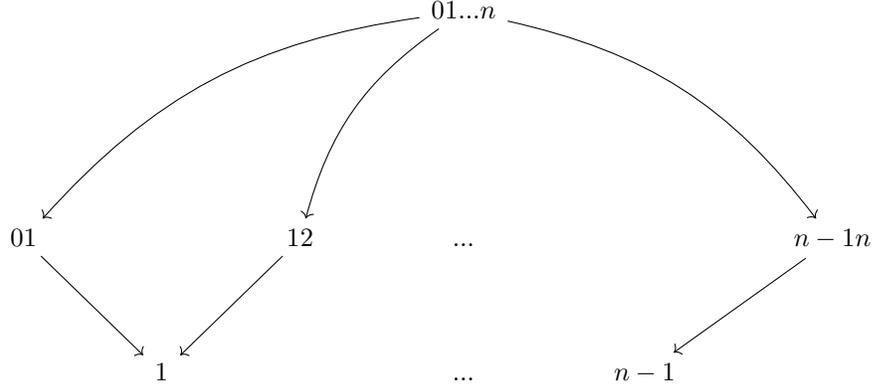
\begin{center}
   \begin{tikzcd}[row sep=0.5in, column sep=0.5in]
     & & & 01...n \arrow[ddlll, bend right=20] \arrow[ddl, bend right=20] \arrow[ddrr, bend left = 20] \\
    \\
    01 \arrow[dr] & & 12 \arrow[dl] & ... & & n-1n \arrow[dl] \\
    & 1 & & ... & n-1 & 
   \end{tikzcd}
  \end{center}
 \end{remone}

 \begin{remone} \label{Rem tildefXn}
  Let $X$ be a simplicial object in $\C$. There is a map of CSS $\n g(n) \to \C$ defined as follows:
  Each object of the form $i$ is mapped to $X_0$. Each object of the form $ij$ is mapped to $X_1$. Each map of the form $0^*$ is mapped to $s: X_1 \to X_0$
  and each map of the form $1^*$ is mapped to $t: X_1 \to X_0$. We will name this map $f_X(n): \n g(n) \to \C$.
  
  The map $f_X(n)$ can  be extended to a map $\tilde{f}_X(n): \n cg(n) \to \C$ as follows. We map the object $01...n$ to the object $X_n$ and each map 
  $a_{ij}$ is mapped to the map $\alpha^*_{i}: X_n \to X_1$. The functoriality then uniquely determines the target of $a_i$.
 \end{remone}
 
 \begin{remone} \label{Rem Segal Diagram}
  Using the depiction of the category above, we can depict this diagram in $\C$ as the following. 
  
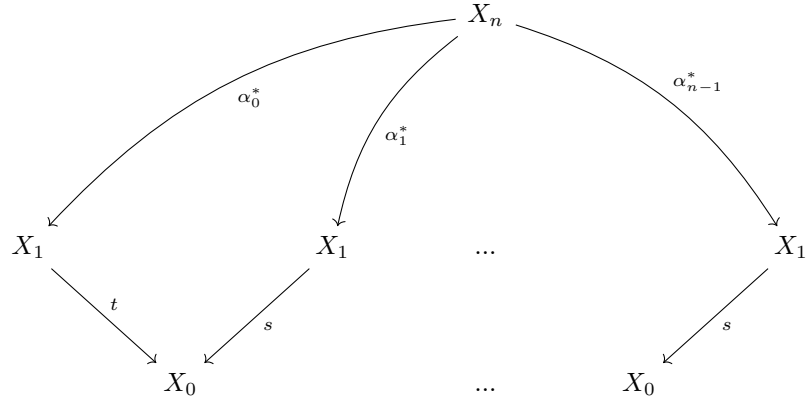
\begin{center}
   \begin{tikzcd}[row sep=0.5in, column sep=0.5in]
     & & & X_n \arrow[ddlll, bend right=20, "\alpha_0^*"] \arrow[ddl, bend right=20, "\alpha_1^*"] \arrow[ddrr, bend left = 20, "\alpha_{n-1}^*"] \\
    \\
    X_1 \arrow[dr, "t"] & & X_1 \arrow[dl, "s"] & ... & & X_1 \arrow[dl, "s"] \\
    & X_0 & & ... & X_0 & 
   \end{tikzcd}
  \end{center}

 \end{remone}
 
 \begin{defone} \label{Def Segal Obj}
  A simplicial object $X$ is a {\it Segal object} if for every $n \geq 2$ the cone $\tilde{f}_X(n)$ is a limit cone for the map $f_X(n)$.
  Informally, we can say the map 
  $$ (\alpha_0^*, ... , \alpha_{n-1}^*) : X_n \xrightarrow{\simeq} X_1 \underset{X_0}{\times} X_1 \underset{X_0}{\times} ... \underset{X_0}{\times} X_1 $$
  is homotopy equivalence in $\C$.
 \end{defone}
 
 Segal objects have their own higher category.
 
 \begin{defone} \label{Def CSS of Seg Obj}
  We define $Seg(\C)$ as the sub full subcategory of the category $\C^{\Delta{op}}$ generated by Segal objects.
 \end{defone}

%
%
%

 Similar to the case of Segal spaces, Segal objects have their own internal category theory that we will discuss in the next subsection.
 
 
 
 
 \subsection{Category Theory of Segal Objects} \label{Subsec Category Theory of Segal Objects}
 Segal objects have their internal homotopy theory. The key difference is that Segal spaces are 
 Segal object in the 
 complete Segal space of spaces which is homotopically generated by the final object (i.e. there is a well-defined notion of membership). 
 This is not true for an arbitrary 
 complete Segal space and so it is not enough to check conditions only for points, which complicates the definitions to some extent.
 
 \begin{notone}
  For this subsection $T: \Delta^{op} \to \C$ is a fixed Segal object in $\C$.
 \end{notone}

 \begin{defone} \label{Def Obj of Seg Obj}
  We define the \emph{objects of $T$} as objects of the over-CSS
 $$Ob(T) = Ob(\C_{/T_0})$$ i.e. the set of maps into $T_0$. Thus objects of $T$ are the set
 $$ Ob(T) = \{ x \in \C_1: t(x) = T_0 \}. $$
 We say an object $x:D \to T_0$ in $T$ has {\it context} $D$ if the domain of $x$ is $D$. 
 \end{defone}

 \begin{remone} \label{Rem Type of Object}
  Notice that every object has an underlying context. The context plays an important role when we later define morphisms and composition.
  We say an object has no context if the domain is the final object $x: * \to T_0$.
 \end{remone}

 
 \begin{defone} \label{Def Map in T}
  A {\it morphism} in T with context $D$ is a map $f: D \to T_1$
 \end{defone}

 \begin{defone} \label{Def Source Target}
  Let $f: D \to T_1$ be a morphism in $T$. Then we define the {\it source of f} as $sf:D \to T_0$ and the {\it target of f} as $tf: D \to T_0$.
  Notice that the source and target of $f$ has the same context. 
 \end{defone}

 $T_1$ gives us a global way to access morphisms, however, we also want to be able to discuss morphisms
 between two objects.
 
 \begin{defone} \label{Def Dep Product}
  Let $\C$ be a locally Cartesian closed CSS. Moreover, let $D$ be a fixed object. Then the map of CSS
  $$D \times -: \C \to \C_{/D}$$
  has a right adjoint, which we denote by 
  $$ \prod_D : \C_{/D} \to \C$$
  and call it the object of sections. 
 \end{defone}

 \begin{remone}
  For the remainder of this subsection let $\C$ be locally Cartesian closed.
 \end{remone}

 \begin{defone} \label{Def Morphism of Seg obj}
  Let $x,y: D \to T_0$ be two objects with context $D$. We define the mapping object $map_T(x,y)$ as
   $$ map_T(x,y) = \prod_D(x,y)^*T_1.$$
   We can depict the situation in the following diagram 
   \begin{center}
    \begin{tikzcd}[row sep=0.5in, column sep=0.5in]
     map_T(x,y) \arrow[d] & (x,y)^*T_1 \arrow[l, "\prod_D"'] \arrow[r] \arrow[d] & T_1 \arrow[d, "(s \comma t)"] \\
     * & D \arrow[l] \arrow[r, "(x \comma y)"] & T_0 \times T_0
    \end{tikzcd}
   \end{center}

   \par 
  
 \end{defone}
 
 \begin{remone}
  How does this object recover maps with source $x: D \to T_0$ and target $y: D \to T_0$? We have following equivalences in $\C$
  $$map_{\C}(*, map_T(x,y)) \simeq map_{/D}(D, (x,y)^*T_1) \simeq map_{/T_0 \times T_0}(D,T_1)$$
  Thus a map $* \to map_T(x,y)$ is exactly the data of a commuting diagram 
  \begin{center}
   \begin{tikzcd}[row sep=0.5in, column sep=0.5in]
    D \arrow[rr, dashed] \arrow[dr, "(x \comma y)"] & & T_1 \arrow[dl, "(x \comma y)"] \\
    & T_0 \times T_0
   \end{tikzcd}
  \end{center}
  This is exactly the data of a morphism in $T$ which has source $x$ and target $y$.
  \par 
   More generally a map $z \to map_T(x,y)$ over $T_0 \times T_0$ can be characterized as
   \begin{center}
    \begin{tikzcd}[row sep=0.5in, column sep=0.5in]
     z \times D \arrow[rr, dashed] \arrow[dr, "(x \comma y) \pi_2"'] & & T_1 \arrow[dl, "(s \comma t)"]\\
     & T_0 \times T_0 & 
    \end{tikzcd}
   \end{center}
  \par 
  In particular, the definition comes with an evaluation map 
  $$ev: D \times map_T(x,y) \to T_1$$
  over $T_0 \times T_0$.
 \end{remone}

 \begin{remone}
  Notice if $x,y$ have no context then the map $\prod_*$ is just the identity map and so $map_T(x,y) = (x,y)^*T_1$.
  In particular, in this case we can define a mapping object just with limits. 
 \end{remone}


%
%

 \begin{notone} 
  As is customary in category theory, we denote a morphism $f$ with source $x$ and target $y$ as $f:x \to y$. 
 \end{notone}

 
 
 Now that we have a notion of morphisms, we need a notion of composition in a Segal objects.
 For that we need to define the {\it object of compositions}.
 Let $(x_0, x_1, ..., x_n):D \to T_0 \times ... \times T_0 = (T_0)^{n+1}$ be objects with context $D$. 
 We define the object of composition as 
 $$ map_T(x_0,x_1,...,x_n) = \prod_D(x_0,x_1,...,x_n)^*T_n.$$
   We can depict the situation in the following diagram 
   \begin{center}
    \begin{tikzcd}[row sep=0.5in, column sep=0.5in]
     map_T(x_0,x_1,...,x_n) \arrow[d] & (x_0,x_1,...,x_n)^*T_n \arrow[d] \arrow[r, "\prod_D"] \arrow[l] & T_n \arrow[d] \\
     * & D \arrow[l] \arrow[r, "(x_0 \comma x_1 \comma ... \comma x_n)"] & (T_0)^{n+1}
    \end{tikzcd}
   \end{center}
  The maps above preserve limit cones. This means the limiting cone $\tilde{f}_T(n)$ as defined in Remark 
  \ref{Rem tildefXn} can be first pulled back to a limiting cone
  
  
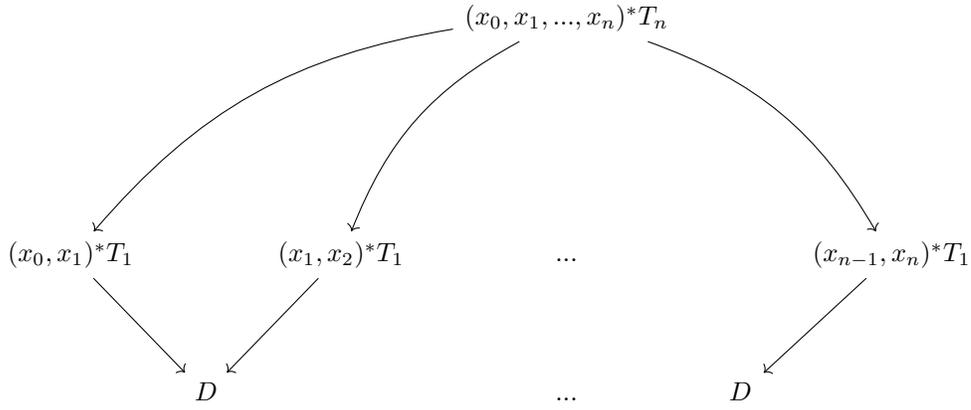
\begin{center}
   \begin{tikzcd}[row sep=0.5in, column sep=0.2in]
     & & & (x_0,x_1,...,x_n)^*T_n 
     \arrow[ddlll, bend right=20] \arrow[ddl, bend right=20] \arrow[ddrr, bend left = 20] \\
    \\
    (x_0,x_1)^*T_1 \arrow[dr] & & (x_1,x_2)^*T_1 \arrow[dl] & ... & & (x_{n-1},x_n)^*T_1 \arrow[dl] \\
    & D & & ... & D & 
   \end{tikzcd}
  \end{center}
  which tells us that we have following pullback diagram.
  $$ (x_0,x_1,...,x_n)^*T_n \xrightarrow{ \ \ \simeq \ \ } (x_0,x_1)^*X_1 \underset{D}{\times} ... \underset{D}{\times} (x_{n-1},x_n)^*X_1$$
  Applying the map $\prod_D$ to this limiting cone we get
  
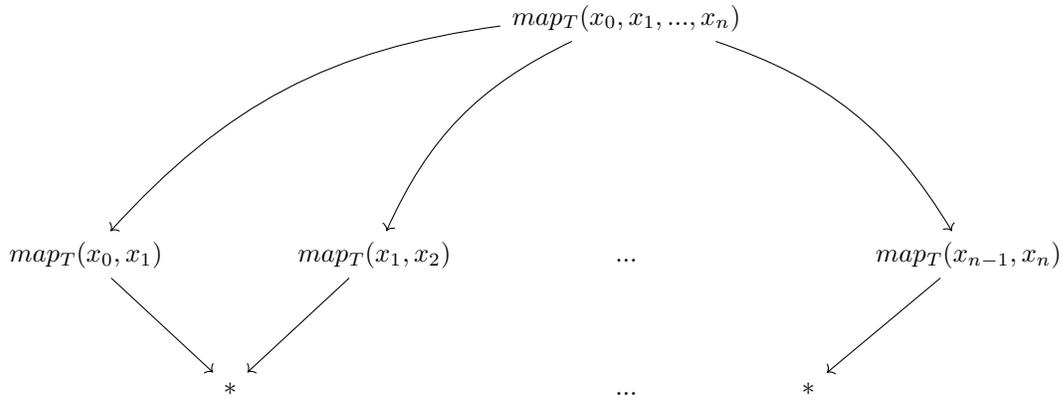
\begin{center}
   \begin{tikzcd}[row sep=0.5in, column sep=0.2in]
     & & & map_T(x_0,x_1,...,x_n) 
     \arrow[ddlll, bend right=20] \arrow[ddl, bend right=20] \arrow[ddrr, bend left = 20] \\
    \\
    map_T(x_0,x_1) \arrow[dr] & & map_T(x_1,x_2) \arrow[dl] & ... & & map_T(x_{n-1},x_n) \arrow[dl] \\
    & * & & ... & * & 
   \end{tikzcd}
  \end{center}
  we will depict this limiting cone with the product map
  $$ (\alpha_0, ..., \alpha_n):  map_T(x_0,x_1,...,x_n) \xrightarrow{ \ \ \simeq \ \ } map_T(x_0,x_1) \times ... \times map_T(x_{n-1},x_n).$$


 \begin{defone} \label{Def Composition}
 Let $x,y,z$ be three objects with context $D$. We have the following equivalence:
 $$ (\alpha_0, \alpha_1): map_T(x,y,z) \xrightarrow{ \ \ \simeq \ \ } map_T(x,y) \times map_T(y,z). $$
 With this in hand we can define the composition of maps. Let $(\alpha_0 , \alpha_1)^{-1}$ be a choice of inverse for the equivalence above. 
 Let $f \in map_T(x,y)$ and $g \in map_T(y,z)$. We define the composition as follows.
 $$ * \xrightarrow{(f,g)} map_T(x,y) \times map_T(y,z) \xrightarrow{(\alpha_0,\alpha_1)^{-1}}
    map_T(x,y,z) \xrightarrow{ \ \ d_1 \ \ } map_T(x,z)$$
 We call this composition map $$ g \circ f: * \to map_T(x,z).$$  
 \end{defone}

 \begin{remone}
  Note that the map is defined only up to a choice of inverse, but the space of inverses is contractible and so any two choices are equivalent. 
 \end{remone}
 
 \begin{remone} \label{Rem Comp like CSS}
  There is another way of understanding the composition.
  Let $Sq((\alpha_0,\alpha_1),d_1)$ be the subspace of $\C_2 \times_{\C_1} \C_2$ generated by squares of the form
  \begin{center}
   \begin{tikzcd}[row sep=0.5in, column sep=0.5in]
    & map_T(x,y) \times map_T(y,z) & \\
    * \arrow[rr] \arrow[ur] \arrow[dr] & & map_T(x,y,z) \arrow[ul, "( \alpha_0 \comma \alpha_1)"', "\simeq"] \arrow[dl, "d_1"] \\
    & map_T(x,z) &
   \end{tikzcd}
  \end{center}
  The map $Sq((\alpha_0,\alpha_1),d_1) \to map_{\C}(*, map_T(x,y)) \times map_{\C}(*,map_T(y,z))$ is a trivial fibration.
  Thus we get following diagram
  \begin{center}
   \begin{tikzcd}[row sep=0.5in, column sep=0.5in]
    Sq((\alpha_0,\alpha_1),d_1) \arrow[d, twoheadrightarrow, "\simeq"] \arrow[r] & map_{\C}(*,map_{\C}(x,z)\\
    map_{\C}(*, map_T(x,y)) \times map_{\C}(*,map_T(y,z))
   \end{tikzcd}
  \end{center}
  A composition of a map $(f,g) \in map_{\C}(*, map_T(x,y)) \times map_{\C}(*,map_T(y,z))$ is a choice of lift to 
  $Sq((\alpha_0,\alpha_1),d_1)$ and the projection
  to $map_{\C}(*,map_{\C}(x,z))$.
  \par 
  This way of defining a composition is exactly in line with the definition of a composition in Segal spaces \cite[5.3]{Re01}.
 \end{remone}

 \begin{defone} \label{Def Identity Map}
  For every object $x: D \to T_0$, there is a morphism $s_0x: D \to T_1$, which we denote by $id_x$. 
  Using the simplicial identities we know that $id_x: x \to x$.
 \end{defone}

 
 \begin{defone} \label{Def Homotopic Morphism}
  Let $f,g: D \to T_1$ be two morphisms with context $D$. We say $f$ is homotopic to $g$ if the corresponding points in the 
  mapping space $map_{\C}(D,T_1)$ are homotopic.
 \end{defone}

 
 Composition behaves as one might expect.
 \begin{propone} \label{Prop Comp Ass Id}
  Let $x,y,z,w: D \to T_1$ be four objects with context $D$.
  For three maps $f \in map_{T}(x,y)$, $g \in map_{T}(y,z)$ and $h \in map_{T}(y,w)$ we have 
  $ ( h \circ g) \circ f \sim h \circ (g \circ f) $ and $ f \circ id_x \sim id_y \circ f \sim f$.
 \end{propone}
 \begin{proof}
 As we defined composition in Segal objects the same way it is defined in a Segal spaces (Remark \ref{Rem Comp like CSS}) 
 the exact same proof generalizes to this setting.
 For a proof of those same properties in a Segal space see \cite[Proposition 5.4]{Re01}.
 \end{proof}

Note that maps of Segal objects have the correct functoriality properties.
 
 \begin{propone}
  Let $F: W \to V$ be a map of Segal objects. For any two objects $x_0,x_1,...,x_n: D \to W$ we get a map 
  $$F(x_0,x_1,...,x_n): map_W(x_0,x_1,...,x_n) \to map_V(Fx_0,FX_1,...,Fx_n)$$
  Moreover, for any two maps $f: x \to y$ and $g: y \to z$ there is an equivalence $F(g) \circ F(f) \sim F(g \circ f)$.
 \end{propone}
 
 \begin{proof}
  The existence of the map $F(x_0,x_1,...,x_n)$ follows from following commutative diagram
  \begin{center}
   \begin{tikzcd}[row sep=0.7in, column sep=0.7in]
    \C_{/W_0^n} \arrow[r, "(x_0 \comma ... \comma x_n)^*"] \arrow[d, "(F^n_0)_!"] & \C_{/D} \arrow[d, equal] \arrow[r, "\prod_D"] & 
    \C \arrow[d, equal] \\
    \C_{/V_0^n} \arrow[r, "(Fx_0 \comma ... \comma Fx_n)^*"] & \C_{/D} \arrow[r, "\prod_D"] & \C
   \end{tikzcd}
  \end{center}
  Using this map in the particular case of $F(x_0,x_1,x_2)$ implies that the $F$ preserves composition.
 \end{proof}
 
 \subsection{Homotopy Equivalences in Segal Objects} \label{Subsec Homotopy Equivalences in Segal Objects}
 The goal in this subsection is to carefully study weak equivalences inside a Segal object.
 Thus for this section let $T$ be a fixed Segal object.

 \begin{defone} \label{Def Homotopy Equiv Comp}
  A map $f: x \to y$ is a {\it homotopy equivalence} if there exists maps $g,h: y \to x$ such that $fg \sim id_x$ and $hf \sim id_y$. 
 \end{defone}
 
 \begin{remone}
  As is customary in higher category theory, composition is not just a fact but actual data. 
  In particular, the definition above implies that there are maps $\iota_1, \iota_2 : D \to T_2$, that can be summarized in following diagrams.
   
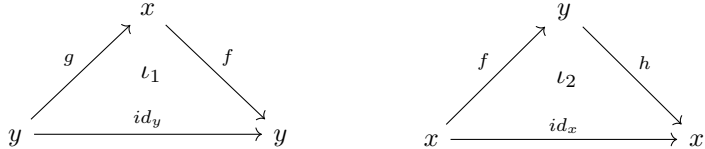
\begin{center}
  \begin{tikzcd}[row sep=0.5in, column sep=0.5in]
   & x \arrow[dr, "f"] & \\
   y \arrow[ur, "g"] \arrow[rr, "id_y", ""{name=U, above}] & & y 
   \arrow[start anchor={[xshift=-10ex, yshift=10ex]}, to=U, phantom, "\iota_1"]
  \end{tikzcd}
  \hspace{0.5in}
  \begin{tikzcd}[row sep=0.5in, column sep=0.5in]
   & y \arrow[dr, "h"] & \\
   x \arrow[ur, "f"] \arrow[rr, "id_x", ""{name=U, above}] & & x 
   \arrow[start anchor={[xshift=-10ex, yshift=10ex]}, to=U, phantom, "\iota_2"]
  \end{tikzcd}
 \end{center}
 \end{remone}

 Now that we have a definition of a homotopy equivalence, we want to define an object that classifies all 
 homotopy equivalences. In \cite{Re01} the space of all homotopy equivalences is simply defined as the subspace of $T_1$, 
 consisting of all homotopy equivalences. Unfortunately we cannot make such a definition as we have no notion of subobjects.
 However, there is an equivalent space that can be defined using a pullback, namely the space $T_{hoeqchoice}$ 
 \cite[Theorem 2.44]{Ra18a}.
 Thus we will generalize this space to the setting of Segal objects.
 
 In order to do that we will give a second equivalent definition that will motivate further constructions.
 First we first need some definitions.
 
 \begin{defone} \label{Def Z3}
  Let $z(3): \n g_T(3) \to \C$ be defined as follows.
  For objects we have $z(3)(01)=z(3)(12)=z(3)(23)= T_1$ and $z(3)(1)=z(3)(2) = T_0$. 
  Moreover, morphisms which have source $1$ map to $s: T_1 \to T_0$ and morphisms which have source $2$ map to $t:T_1 \to T_0$.
  We define $Z^T(3)$ to be the limit cone of this diagram. We can depict this as following diagram.
  
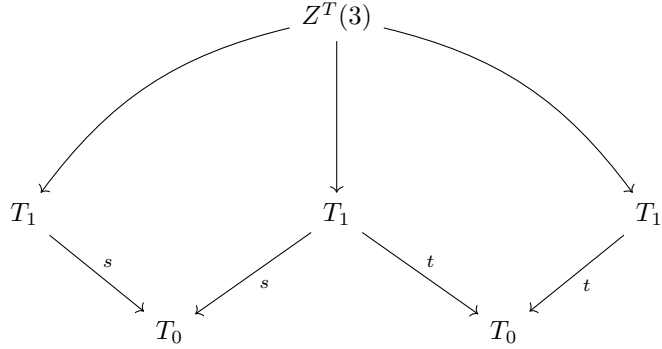
\begin{center}
   \begin{tikzcd}[row sep=0.4in, column sep=0.5in]
     & & Z^T(3) 
     \arrow[ddll, bend right=20] \arrow[dd] \arrow[ddrr, bend left = 20] \\
    \\
    T_1 \arrow[dr, "s"] & & T_1 \arrow[dl, "s"] \arrow[dr, "t"] & & T_1 \arrow[dl, "t"]  \\
    & T_0 & & T_0 & 
   \end{tikzcd}
  \end{center}
 \end{defone}
 This constructions comes with two important simplicial maps:
 $$ (d_1d_3,d_0d_3,d_1d_0) : T_3 \to Z^T(3) $$
 $$ (s_0d_0,id_{T_1},s_0d_1): T_1 \to Z^T(3) $$
 Having these two maps we can give a second characterization of homtopy equivalences in $T$.
 
 \begin{lemone} \label{Lemma Homotopy Equiv Lifting}
  A map $f: D \to T_1$ is a \emph{homotopy equivalence in T} if and only if there
  exists a $2$-cell $\sigma$ that lifts the diagram below. 
  We call the lift $d_1\sigma = \tilde{f}: D \to T_3$ 
 
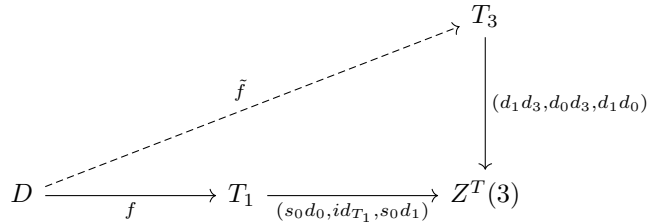
\begin{center}
  \begin{tikzcd}[row sep=0.7in, column sep=0.9in]
     &   & T_3 \arrow[d, "(d_1d_3 \comma d_0d_3 \comma d_1d_0)"] \\
   D \arrow[urr, dashed, "\tilde{f}"] \arrow[r, "f"'] & T_1 \arrow[r, "(s_0d_0 \comma id_{T_1} \comma s_0d_1)"'] & 
   Z^T(3)
  \end{tikzcd}
 \end{center} 
 \end{lemone}
 
 \begin{proof}
  If there is a map $\tilde{f}: D \to T_3$ that lifts $f$, then $d_1d_3\tilde{f}:D \to T_1$ and $d_0d_0\tilde{f}:D \to T_1$ give us the two inverses.
  For the other side, we first need following fact.
  The map $T_3 \to Z^T(3)$ factors through the map $(d_3,d_0): T_3 \to T_2 \times_{T_1} T_2$, which is an equivalence by 
   the Segal condition.
   \par 
   By the definition of limits, we have an equivalence $D \to T_2 \times_{T_1} T_2$ over $Z^T(3)$ which is determined by 
   following three maps:
   \begin{center}
    \begin{tikzcd}[row sep=0.5in, column sep=0.4in]
     D \arrow[dr, "(f \comma s_0d_0f)"'] \arrow[rr, dashed, "\iota_1"] & & T_2 \arrow[dl, "(d_1 \comma d_0)"] \\
     & T_1 \underset{T_0}{\times} T_1& 
    \end{tikzcd}
    \begin{tikzcd}[row sep=0.5in, column sep=0.4in]
     D \arrow[dr, "f"'] \arrow[rr, dashed, "f"] & & T_1 \arrow[dl, "id_{T_1}"] \\
     & T_1 & 
    \end{tikzcd}
     \begin{tikzcd}[row sep=0.5in, column sep=0.4in]
     D \arrow[dr, "(s_0d_1f \comma f)"'] \arrow[rr, dashed, "\iota_2"] & & T_2 \arrow[dl, "(d_2 \comma d_1)"] \\
     & T_1 \underset{T_0}{\times} T_1 & 
    \end{tikzcd}
   \end{center}
   that agree with each other appropriately. However, this is exactly the two composition $\iota_1, \iota_2: D \to T_2$.
 \end{proof}

 Having a second criterion for equivalences, we can finally give the main definition.
 
 
 \begin{defone} \label{Def Object of Equiv}
  We define the {\it object of equivalences} as the following pullback
  \begin{center}
  \pbsq{T_{hoequiv}}{T_3}{T_1}{Z^T(3)}{e}{U}{
  (d_1d_3 \comma d_0d_3 \comma d_1d_0)}{(s_0d_0 \comma id \comma s_0d_1)}
 \end{center}
 \end{defone}

 We have the following facts about homotopy equivalences in Segal objects:
 \begin{lemone}
  A map  $f:D \to T_1$, is a homotopy equivalence if and only if there exists a $2$-cell $\sigma$ of the form.
  \begin{center}
   \begin{tikzcd}[row sep=0.6in, column sep=0.6in]
    & T_{hoequiv} \arrow[d, "U"]  \\
    D \arrow[ur, "witeq(f)", dashed] \arrow[r, "f"', ""{name=U, above}] & T_1 
    \arrow[start anchor={[xshift=0ex, yshift=8ex]}, to=U, phantom, "\sigma"]
   \end{tikzcd}
  \end{center}
  We denote the map $d_1\sigma: D \to T_{hoequiv}$ as $witeq(f)$, as it is a witness for the fact that $f$ is an equivalence.
 \end{lemone}
 
 \begin{proof}
  If $f$ is a homotopy equivalence then there is a $\tilde{f}:D \to T_3$ making the diagram in Lemma \ref{Lemma Homotopy Equiv Lifting} commute, 
  so by the universality of pullbacks
  we get a two cell  $\sigma$ that lifts the desired diagram. 
  On the other hand if $f$ factors then the map $e \circ witeq(f)$ is exactly the lift 
  which makes $f$ into a homotopy equivalence, by the previous lemma.
 \end{proof}
 
 These results give us two interesting homotopy equivalences in $T$.
 
 \begin{lemone}
  The map $s_0: T_0 \to T_1$ is a homotopy equivalence with context $T_0$. 
 \end{lemone}

 \begin{proof}
  Just to illuminate how the two approaches compare, we give two proofs.
  \par 
  First, notice that $s_0: T_0 \to T_1$ is the identity map for the object $id:T_0 \to T_0$ (with context $T_0$).
  By Proposition \ref{Prop Comp Ass Id} we have $s_0 \circ s_0 \sim s_0$ (where $s_0 \circ s_0$ is the composition inside $T$) 
  and so $s_0$ is a homotopy equivalence by 
  Definition \ref{Def Homotopy Equiv Comp}.
  \par 
  Second, notice that the unique map $T_0 \to T_3$ makes the diagram in Lemma \ref{Lemma Homotopy Equiv Lifting} commute and so 
  $s_0$ is a homotopy equivalence.
 \end{proof}

 \begin{corone}
  The map $i: T_{hoequiv} \to T_1$, is a homotopy equivalence in $T$. 
 \end{corone}
 

 Notice we defined $T_{hoequiv}$ as a pullback which means following lemma holds. 
 \begin{lemone} \label{Lemma Thoequiv universal prop}
  Let $T_{h} \to T_1$ satisfy the property that any morphism $D \to T_1$ is an equivalence if and only if it lifts to $T_h$.
  Then $T_h$ is equivalent to $T_{hoequiv}$.
 \end{lemone}

 There is also the following helpful lemma to determine whether a map is an equivalence.
 
 \begin{lemone}
  The morphism $f:D \to T_1$ is homotopy equivalence in $T$ if and only if every precomposition $fg:E \to T_1$ is a homotopy equivalence in $T$.
 \end{lemone}
 
 \begin{proof}
  If any precomposition is a equivalence then so in particular is $f id_D = f$. If $f$ is a homotopy equivalence then $f$ factors with codomain
  $T_{hoequiv}$ and so does every precomposition.
 \end{proof}
 
%
 
 Up until now we have defined an equivalence in terms of certain liftings. One question that comes up is the issue of uniqueness.
 Can a homotopy equivalence have several lifts to $T_{hoequiv}$? In the case of Segal spaces this is clearly not possible as we define 
 $T_{hoequiv}$ as a subspace of $T_1$ and so the map between them is injective. 
 As our definition of $T_{hoequiv}$ differs we have to confirm that such liftings are unique in the correct sense.
 \begin{lemone}
  The map $U: T_{hoequiv} \to T_1$ is (-1)-truncated
 \end{lemone}
 \begin{proof}
  The map is ($-1$)-truncated if and only if in the following pullback square
  \begin{center}
   \pbsq{T_{hoequiv} \underset{T_1}{\times} T_{hoequiv}}{T_{hoequiv}}{T_{hoequiv}}{T_1}{\pi_1}{\pi_2}{U}{U}
  \end{center}
 $T_{hoequiv} \times_{T_1} T_{hoequiv}$ is equivalent to $T_{hoequiv}$.
 We will prove that by showing that the map $U\pi_1 : T_{hoequiv} \times_{T_1} T_{hoequiv} \to T_1$ satisfies the universal property described in
 Lemma \ref{Lemma Thoequiv universal prop}. 
 \par 
 A map $f: D \to T_1$ is an equivalence if and only if we can complete following diagram 
 \begin{center}
  \begin{tikzcd}[row sep=0.5in, column sep=0.7in]
   & T_{hoequiv} \arrow[d, "U"] \\
   D \arrow[r, "f"] \arrow[dr, "witeq(f)"', dashed] \arrow[ur, "witeq(f)", dashed] & T_1 \\
   & T_{hoequiv} \arrow[u, "U"']
  \end{tikzcd}
 \end{center}
 which is exactly the data of a map $D \to T_{hoequiv} \times_{T_1} T_{hoequiv}$. Thus $f: D \to T_1$ is an equivalence if and only if 
 it lifts to a map $(witeq(f), witeq(f)):D \to T_{hoequiv} \times_{T_1} T_{hoequiv}$, which means we are done.


 \end{proof}
 
 The lemma above has following important corollary.
 
 \begin{corone} \label{Cor Uniqueness of inverse lifts}
  For any map $f: D \to T_1$, the internal mapping object $map_{/T_1}(D, T_{hoequiv})$ is either empty or contractible.
 \end{corone}
 
 Thus either a morphism $f: D \to T_1$ is not an equivalence and thus has no lift, or it is an equivalence and the space of such lifts is contractible.
 This is telling us that if a morphism has inverses then they are determined uniquely (up to homotopy) by the morphism itself.
 

 For two objects $x,y: D \to T_0$, we defined $map_T(x,y)$ as the mapping object which contains the data of maps from $x$ to $y$.
 We now want to repeat the same procedure for homotopy equivalences.
 
 \begin{defone} \label{Def Hoequiv xy}
 Let $\C$ be locally Cartesian closed.
  Let $x,y : D \to T_0$ be two objects, we define the {\it object of equivalences} $hoequiv_T(x,y)$ as
   $$hoequiv_T(x,y) = \prod_D (x,y)^*T_{hoequiv}$$
   This can be captured in following diagram.
  \begin{center}
    \begin{tikzcd}[row sep=0.5in, column sep=0.5in]
     hoequiv_T(x,y) \arrow[d] & (x,y)^*T_{hoequiv} \arrow[d] \arrow[l, "\prod_D"'] \arrow[r] & T_{hoequiv} \arrow[d, "(s \comma t)U"] \\
     * & D \arrow[l] \arrow[r, "( x \comma y)"] & T_0 \times T_0
    \end{tikzcd}
   \end{center}   
   It comes with an evaluation map 
   $$ev: hoequiv_T(x,y) \times D \to T_{hoequiv}$$
 \end{defone}
 
 \begin{remone}
  The map $U:T_{hoequiv} \to T_1$ induced a map $U(x,y):hoequiv_T(x,y) \to map_T(x,y)$. As right adjoints preserve truncation levels, this map is still 
  ($-1$)-truncated.
 \end{remone}

%
%

 \subsection{Definition of a Complete Segal Object} \label{Definition of a Complete Segal Object}
 
 As of now we have discussed homotopy equivalence in Segal objects.
 However, by doing so we have created a possible ambiguity. 
 Two objects $x,y: D \to T_0$ can now be equivalent in more than one way. 
 First, we can have an equivalence (in the ambient category $\C$) $h: D \to D$ such that $yh$ is equivalent to $x$.
 Second, the object of equivalences $hoequiv_T(x,y)$ can be non-trivial (in the sense that a lift to $T_{hoequiv}$ exists).
 \par 
 We want to make sure that these two conditions do coincide. 
 The way to achieve this is via the completeness condition.
 
 \begin{defone}
  We say a Segal object $W$ is complete Segal object if the map $s_0: W_0 \to W_{hoequiv}$ is an equivalence in $\C$.
 \end{defone}

 
 Given the definition of $W_{hoequiv}$ we have following equivalent way of defining a complete Segal object.
 
 \begin{lemone}
  A Segal object $W$ is complete if and only if the following is a pullback square in $\C$.
 \begin{center}
    \pbsq{W_0}{W_3}{W_1}{Z^W(3)}{}{}{}{}
   \end{center}
 \end{lemone}
 
 Similar to Segal objects, complete Segal objects also have their own category theory.
 
 \begin{defone} \label{Def CSS of CSO}
  We define $CSO(\C)$ as the full subcategory of $\C^{\Delta^{op}}$ generated by complete Segal objects.
 \end{defone}

 Up to here we have defined a complete Segal object as something that plays the role of an ``internal higher category".
 Our next goal should be to develop standard categorical tools for complete Segal objects.
 In particular, we want to be able to talk about limits and colimits or adjunctions in a complete Segal object.
 However, such discussions can become quite difficult as the ambient category $\C$ might lack many pleasant properties that we
 can find in the category of spaces. 
 \par
 One approach is to impose enough conditions on $\C$ to get the desired result.
 For this approach see \cite{RS17}, where
 Riehl and Shulman use techniques from type theory to be able to do internal higher category theory. 
 In particular, they define an internal version of the free arrow $F(1)$, which allows them to do many important categorical constructions.
 \par 
 Another approach is to embed complete Segal objects in a larger environment that allows us to do categorical constructions externally.
 The key motivation here is the Yoneda lemma, which shows that we can embed any higher category in the category of presheaves, 
 or in a higher category that models presheaves, namely right fibrations. The category of right fibrations has many pleasant properties
 and in particular has a model structure, which gives us concrete ways to do many constructions.
 The goal is to generalize the statement above to complete Segal objects. We embed complete Segal objects in a generalization of 
 right fibrations, namely Cartesian fibrations and use the model structure on Cartesian fibrations to develop
 the category theory of complete Segal objects.
 
\section{Representable Cartesian Fibrations} \label{Sec Representable Cartesian Fibrations}
 In classical category theory we have the Yoneda embedding 
 $$ \mathscr{Y}: \mathscr{D} \to Fun(\mathscr{D}^{op}, \set)$$
 which embeds a category $\mathscr{D}$ in the category of presheaves. This uses the fact that every object $d$ gives rise to a 
 {\it representable presheaf} $Hom_{\mathscr{D}}(-,d)$.
 In higher category theory presheaves can be quite complicated because
 of functoriality issues and so we use right fibrations to model presheaves valued in spaces.
 Using the same Yoneda lemma argument for every object in a higher category we get a {\it representable right fibration}.
 \par 
 We want to use the same argument for complete Segal objects. However, a complete Segal object has more information than
 just one object and so it will give us a presheaf valued in higher categories, rather than spaces.
 Presheaves valued in higher categories are modeled by {\it Cartesian fibrations}. 
 Thus our goal is to understand Cartesian fibrations built out of complete Segal objects,
 so called {\it representable Cartesian fibrations}.
 \par 
 The theory of Cartesian fibrations for complete Segal spaces is carefully studied in \cite{Ra17b} and will serve as our main reference
 in this section.
 Thus, in the first subsection we review the most important results that we need later on.
 Then we use the results in the coming subsections to study representable Cartesian fibrations
 which then can be used to learn more about complete Segal objects.
 
 
 \begin{remone}
  For this section $\C$ is a CSS with finite limits. Most definitions in this section hold in a more general setting (see \cite{Ra17b}),
  however, we will focus on the case we need. 
 \end{remone}

 \subsection{Cartesian Fibrations} \label{Subsec Cartesian Fibrations}
 In this subsection we review the important definitions with regard to Cartesian fibrations that we will need to study 
 complete Segal objects. For more details see \cite{Ra17a} and \cite{Ra17b}.
 
 \begin{defone} \label{Def Left Fibration}
 \cite[Definition 3.1]{Ra17a}
 A map $p: \L \to \C$ is called {\it left fibration} if it is a Reedy fibration and the following is a 
 homotopy pullback square:
 \begin{center}
  \pbsq{\L^{F(1)}}{\L}{\C^{F(1)}}{\C}{s}{p^{F(1)}}{p}{s}
 \end{center}
\end{defone}
 
 Left fibrations come with a model structure.
  \begin{theone} \label{The Covariant Model Structure}
  \cite[Theorem 3.14]{Ra17a}
  There is a unique model structure on the category $s\s_{/\C}$, , called the covariant model structure
  and denoted by $(s\s_{/\C})^{cov}$, which satisfies the following conditions:
  \begin{enumerate}
   \item It is a simplicial model category
   \item The fibrant objects are the left fibrations over $\C$
   \item Cofibrations are monomorphisms
   \item A map $f: A \to B$ over $\C$ is a weak equivalence if 
   $$map_{s\s_{/\C}}(B,\L) \to map_{s\s_{/\C}}(A,\L)$$ 
   is an equivalence for every left fibration $\L \to X$.
   \item A weak equivalence (covariant fibration) between fibrant objects is a level-wise equivalence (Reedy fibration). 
  \end{enumerate}
 \end{theone}
  
  Note that the definition is not symmetric and so we have following definition.
   \begin{defone} \label{Def Right Fibration}
    \cite[Definition 3.21]{Ra17a}
  A map $p: \R \to \C$ is called {\it right fibration} if it is a Reedy fibration and the following is a 
  homotopy pullback square:
  \begin{center}
   \pbsq{\R^{F(1)}}{\R}{\C^{F(1)}}{\C}{t}{p^{F(1)}}{p}{t}
  \end{center}
  \end{defone}
  \begin{remone}
    Similar to the previous case this fibration comes with its own model structure, which is called the {\it contravariant model structure}.
  \end{remone}

%
%
  For more details on left fibrations and it's relevant properties see \cite[Chapter 3]{Ra17a}.

  Left fibrations model maps into spaces.
  The next step is it to generalize everything to the level of presheaves valued in higher categories.
  However, before we can do so we have to expand our playing field.
  
 \begin{defone} \label{Def Bisimplicial Spaces}
  Let $ss\s$ be the category with objects bisimplicial spaces. 
 \end{defone}
 
 The category of bisimplicial spaces has its own version of Reedy model structure, which we call the biReedy model structure.
 
 \begin{theone} \label{The BiReedy Model Structure}
  There is a model structure on $ss\s$, called the biReedy model structure, defined as follows:
  \begin{itemize}
   \item[W] A weak equivalence is a level-wise Reedy equivalence of simplicial spaces.
   \item[C] A cofibration is an inclusion.
   \item[F] A fibration is a map that satisfies the right lifting property with respect to trivial cofibrations.
  \end{itemize}
 \end{theone}

 Having bisimplicial spaces we can define three fibrations that give us a model of a functor.
 
 \begin{defone} \label{Def Reedy Right Fib}
  \cite[Definition 4.14]{Ra17b}
  Let $p:\R \to \C$ be a bisimplicial space over $\C$.
  We say $p$ is a Reedy right fibration if it is a biReedy fibration and it is a level-wise right fibration.
 \end{defone}

 \begin{defone} \label{Def Segal Cartesian Fibration}
 \cite[Definition 7.2]{Ra17b}
 A map $p: \R \to \C$ is called a {\it Segal Cartesian fibration} if it satisfies the following conditions:
 \begin{enumerate}
  \item It is a Reedy right fibration.
  \item It satisfies the Segal condition, meaning the map 
  $$\R_n \to \R_1 \underset{\R_0}{\times} ... \underset{\R_0}{\times} \R_1$$
  is a Reedy equivalence of simplicial spaces.
 \end{enumerate}
\end{defone} 
 
 \begin{defone} \label{Def Cartesian Fibration}
 \cite[Definition 7.15]{Ra17b}
 A map $p: \R \to \C$ is called a {\it Cartesian fibration} if it is a Segal Cartesian fibration and satisfies the following conditions:
 \begin{itemize}
  \item {\it Completeness Condition:} The map 
  $$\R_0 \to \R_3 \underset{\R_1 \underset{\R_0}{\times} \R_1 \underset{\R_0}{\times} \R_1}{\times} \R_1$$
  is a Reedy equivalence of simplicial spaces.
 \end{itemize}
\end{defone} 
 
 (Segal) Cartesian fibrations come with model structures as well.
  \begin{theone} \label{The Segal Cartesian Model Structure}
  \cite[Theorem 7.3, Theorem 7.16]{Ra17b}
  There is a unique model structure on the category $ss\s_{/\C}$, called the (Segal) Cartesian model structure
  and denoted by $(ss\s_{/X})^{(Seg)Cart}$, which satisfies the following conditions:
  \begin{enumerate}
   \item It is a simplicial model category
   \item The fibrant objects are the (Segal) Cartesian fibrations over $\C$
   \item Cofibrations are monomorphisms
   \item A map $f: A \to B$ over $X$ is a weak equivalence if 
   $$map_{ss\s_{/X}}(B,C) \to map_{ss\s_{/X}}(A,C)$$ 
   is an equivalence for every (Segal) Cartesian fibration $C \to X$.
   \item A weak equivalence ((Segal) Cartesian fibration) between fibrant objects is a level-wise equivalence (biReedy fibration). 
  \end{enumerate}
 \end{theone}

 We have following recognition principle for (Segal) Cartesian fibration.
 
 \begin{defone}
  There is a map $(i_{\varphi})_*: ss\s \to s\s$ defined as 
  $$(i_{\varphi})_*(X)_{nl} = X_{n0l}$$
 \end{defone}

 \begin{theone}
  \cite[Corollary 7.4, Corollary 7.18]{Ra17b}
  Let $\R \to \C$ be a Reedy right fibration. The following are equivalent.
  \begin{enumerate}
   \item $\R \to \C$ is a Segal Cartesian fibration (Cartesian fibration).
   \item $(i_{\varphi})_*(\R)$ is a Segal space (CSS).
   \item $(i_{\varphi})_*(\R \times_{\C} F(0))$ is a Segal space (CSS) for each object $c$ in $\C$.
  \end{enumerate}
 \end{theone}

 \subsection{Representable Reedy Right Fibrations} 
 \label{Subsec Representable Reedy Right Fibrations}
 There is a well-established theory of representable right fibrations that is motivated by theory of representable presheaves.
 In this subsection we review representable right fibrations and show how they generalize to representable Cartesian fibrations.
 
 \begin{defone}
  Let $ c$ be an object in $\C$, we define the over-category $\C_{/c}$ as 
  $$\C_{/c} = \C^{F(1)} \underset{\C}{\times} F(0)$$
  thus we can think of it as the subcategory of the arrow category $\C^{F(1)}$ generated by the arrows which have target $c$.
  
 \end{defone}
 It comes with a natural projection map $s:\C_{/c} \to \C$ that sends each arrow to its source. This map is quite important.
 \begin{theone}
  \cite[Example 3.11]{Ra17a}
  For any object $c$ in $\C$, the map $s: \C_{/c} \to \C$ is a right fibration.
 \end{theone}

 \begin{theone}
 \cite[Theorem 4.2]{Ra17a}
  The map $F(0) \to \C_{/c}$ that sends the point to the identity map $id_c: c \to c$ is a contravariant equivalence.
 \end{theone}

 This in particular has following important corollary, we which we can think of as the ``Yoneda Lemma for right fibrations".
 
 \begin{corone}
  Let $\R \to \C$ be a right fibration over $\C$ and $c$ an object in $\C$. Then we have an equivalence 
  $$Map_{\C}(\C_{/c},\R) \xrightarrow{ \ \ \simeq \ \ } Map_{\C}(F(0),\R) \simeq \Delta[0] \underset{\C_0}{\times} \R_0$$
 \end{corone}
 
 This corollary justifies following definition.
 
 \begin{defone}
  We say a right fibration $\R \to \C$ is representable if it is equivalent to a right fibration of the form $\C_{/c} \to \C$, for some
  object $c$ in $\C$.
 \end{defone}

 Our next goal is it to generalize all of these results to simplicial objects and Reedy right fibrations.
 So, we first generalize over-categories and construct a Cartesian fibration out of simplicial objects.

 \begin{defone}
  Let $X_{\bullet}$ be a simplicial object in $\C$. We define $\C_{/X_{\bullet}}$ as following bisimplicial space
  $$(\C_{/X_{\bullet}})_k = \C_{/ \pi^i_kX}.$$ 
  Here  $(\pi^i)_k: (\Delta^{op})_{k/} \to \Delta^{op}$ is the projection map.
 \end{defone}

 \begin{theone}
  \cite[Definition 5.27, Proposition 5.29]{Ra17b}
  $\C_{/X_{\bullet}} \to \C$ is a Reedy right fibration. 
 \end{theone}

 Such Reedy right fibrations even come with their own Yoneda Lemma.
 
 \begin{theone}
  \cite[Theorem 5.34]{Ra17b}
  Let $X_{\bullet},Y_{\bullet}$ be two simplicial objects in $\C$. Then we have a map 
  $$Map_{\C}(\C_{/X_{\bullet}},\C_{/Y_{\bullet}}) \xrightarrow{ \ \ \simeq \ \ } map_{\C^{\Delta^{op}}}(X_{\bullet},Y_{\bullet})$$
  that is a trivial Kan fibration.
 \end{theone}
 
 This justifies a new definition.
 
 \begin{defone}
  A Reedy right fibration is {\it representable} if it is equivalent to a 
  Reedy right fibration of the form $\C_{/X_{\bullet}}$ for some simplicial object $X_{\bullet}$.
 \end{defone}

 Up until here we have defined Cartesian fibrations which model presheaves valued in higher categories
 and constructed representable Reedy right fibrations which build functors valued in simplicial spaces out of simplicial objects.
 The next goal is to combine these two approaches and build representable (Segal) Cartesian fibrations
 out of (complete) Segal objects.
 
 
 \subsection{Representable Cartesian fibrations out of Complete Segal Objects}
 \label{Subsec Representable Cartesian fibrations out of Complete Segal Objects}
 In the past subsection we showed how we can use simplicial objects to build Reedy right fibrations.
 In this section we specialize this approach to show we can build (Segal) Cartesian fibrations.
 
 \begin{theone}
  Let $W$ be a Segal object in $\C$. Then $\C_{/W}$ is a Segal Cartesian fibration.
 \end{theone}
 
 \begin{proof}
  We already know that $\C_{/W}$ is a Reedy right fibration.
  It suffices to show that the fiber over each point is actually a Segal space.
  Let $D$ be an object in $\C$. We have to show that the simplicial space $(i_{\varphi})_*(\C_{/W} \times_{\C} F(0))$ is a Segal space.
  The Reedy right fibration condition implies that it is Reedy fibrant. Thus we only have to 
  show that it satisfies the Segal condition.
  \par 
  By definition we have an equivalence $((\C_{/W})_k \times_{\C} F(0))_0 \simeq map_{\C}(D,W_k)$.
  Moreover, we have an equivalence
  $$map_{\C}(D,W_k) \to map_{\C}(D,W_1) \underset{map_{\C}(D,W_0)}{\times} ...  \underset{map_{\C}(D,W_0)}{\times} map_{\C}(D,W_1).$$
  This follows from the fact that the diagram in Remark \ref{Rem Segal Diagram} is a limiting cone, which is preserved by the mapping spaces.
  However, this implies that the map 
  $$((\C_{/W})_k \times_{\C} F(0))_0 \to ((\C_{/W})_1 \times_{\C} F(0))_0 \underset{((\C_{/W})_0 \times_{\C} F(0))_0}{\times} ... 
  \underset{((\C_{/W})_0 \times_{\C} F(0))_0}{\times} ((\C_{/W})_1 \times_{\C} F(0))_0 $$
  is a trivial Kan fibration, which shows that $(i_{\varphi})_*(\C_{/W} \times_{\C} F(0))$ is a Segal space and hence we are done.
 \end{proof}

 \begin{theone}
  Let $W$ be a complete Segal object in $\C$. Then $\C_{/W}$ is a Cartesian fibration.
 \end{theone}
 
 \begin{proof}
  From the previous theorem we already know that $\C_{/W}$ is a Segal Cartesian fibration.
  In order to show that it is Cartesian we have to show that every fiber satisfies the completeness condition.
  However, this again follows from the fact that $((\C_{/W})_k \times_{\C} F(0))_0 \simeq map_{\C}(D,W_k)$
  paired with the fact that the $W$ satisfies the completeness condition, which is preserved
  when we move to mapping spaces.
 \end{proof}

 \begin{defone}
  A (Segal) Cartesian fibration is {\it representable} if it is equivalent to a (Segal) Cartesian fibration of the 
  form $\C_{/W}$ for some (complete) Segal object $W$.
 \end{defone}
 
 \begin{corone}
  Let $W,V$ be two CSO in $\C$. Then we have a map 
  $$Map_{\C}(\C_{/W},\C_{/V}) \xrightarrow{ \ \ \simeq \ \ } map_{CSO(\C)}(W,V)$$
  that is a trivial Kan fibration.
 \end{corone}

 This has following very important corollary
 
 \begin{corone}
  Two CSOs $W$ and $V$ in $\C$ are equivalent in $CSO(\C)$ if and only if the corresponding Cartesian fibrations 
  $\C_{/W}$ and $\C_{/V}$ are equivalent.
 \end{corone}

 Thus we can study complete Segal objects via their Cartesian fibrations. However in 
 Theorem \ref{The Segal Cartesian Model Structure} we showed Cartesian fibrations come with a
 very well behaved model structure and so we can apply tools from model category theory to help.
 
 It is valuable to see how the results about complete Segal objects translate when we look at it from the perspective 
 of Cartesian fibrations.
 

 \begin{defone}
   Let $D$ be an object and $W$ a CSO in $\C$. We define the simplicial space $map_{\C}(D,W)$ as the following pullback diagram
   \begin{center}
    \pbsq{map_{\C}(D,W)}{(i_{\varphi})_*(C_{/W})}{F(0)}{\C}{}{}{}{D}
   \end{center}
 \end{defone}
 
 \begin{notone}
  We will also denote $map_{\C}(D,W)$ as $W_D$ in order to simplify notations.
 \end{notone}
 
 \begin{remone}
  Denoting the fiber by $map_{\C}(D,W)$ is quite reasonable as we do have an equivalence of spaces 
  $$((i_{\varphi})_* F(0) \underset{\C}{\times} \C_{/W})_k \simeq map_{\C}(D,W_k)$$
 \end{remone}

 
 \begin{remone}
  We have now completely justified our previous intuition. The Cartesian fibration $\C_{/W}$ models the functor 
  $$map_{/C}(-,W): \C^{op} \to CSS$$
  that takes each object $D$ to the complete Segal space defined as $map_{\C}(D,W)$.
 \end{remone}
 
 Now that we have established these facts, we can translate between complete Segal spaces and complete Segal objects.
 
 \begin{remone} \label{Rem CSO to CSS}
  An object $x: D \to W$ with context $D$ corresponds to an object in the CSS $map_{\C}(D,W)$
  as it is a point $x \in map_{\C}(D,W)_0$.
  Similarly, a morphism $f: D \to W$ corresponds to a morphism in the CSS $map_{\C}(D,W)$.
  Moreover, a homotopy equivalence $f: D \to W_{hoequiv}$ corresponds to a point in the space $map_{\C}(D,W)_{hoequiv}$. 
  \par
  Extending this argument, for two objects $x,y: D \to W$, we have an equivalence 
  $$map_{\C}(*,map_{W}(x,y)) \xrightarrow{ \ \ \simeq \ \ } map_{W_D}(x,y)$$
 \end{remone}
 
  Thus the study of the complete Segal object $W$ corresponds to the study of the collective complete Segal spaces 
  $W_D$ for all objects $D$.
  Having established a strong, rigorous connection between complete Segal spaces and complete Segal objects, 
  we can now prove higher categorical results about complete Segal objects using what we know about complete Segal spaces.
 
 \subsection{Building Simplicial Objects} \label{Subsec Building Simplicial Objects}
 In this subsection we make use of our correspondence between Segal objects and Segal Cartesian fibrations
 and develop techniques that allow us to build new simplicial objects.
 Then we specialize to the case of Segal objects. 
 
 \begin{theone} \label{The Building Simp Obj}
  Let $X_0, X_1, ...$ be a sequence of objects in $\C$ such that there exists a Reedy right fibration $\R \to \C$ 
  such that $\R_n \simeq \C_{/X_n}$. Then there exists a simplicial object $\hat{X}: \Delta^{op} \to \C$ such that 
  $\hat{X}_n \simeq X_n$.
 \end{theone}
 
 \begin{proof}
  By assumption $\R$ is a representable Reedy right fibration, which implies that there is a map $(\Delta^{op})_{/ \bullet} \to \R$ 
  over $\C$. We will use the map $q:\Delta^{op}_{/ \bullet} \to \C$ to define our simplicial object. 
  We define $\hat{X}: \Delta^{op} \to \C$ as follows 
  $$\hat{X}_n(f_1,...,f_n) = q_{t(f_n)}(f_1,...,f_n)$$
  Here we used the fact that the $n$-cell $(f_1,...,f_n)$ in $\Delta^{op}$ gives us an $n$-cell $(f_1,...,f_n)$ in $\Delta^{op}_{/ t(f_n)}$.
  The functoriality of the construction follows from the functoriality of $q$.
 \end{proof}
 
 \begin{remone}
  The essence of the proof is a classical Yoneda style argument. The maps have to exist in $\C$ because they exist at the level of the presheaf
  represented by the objects. 
 \end{remone}
 
 \begin{theone} \label{The Rep Segal Cart}
  Let $\C$ be a CSS with finite limits and let $\R \to \C$ be a Segal Cartesian fibration.
  Then $\R$ is representable if and only if $\R_0$ and $\R_1$ are representable.
 \end{theone}

 \begin{proof}
  One side is just a special case. So, let us assume $\R_0$ and $\R_1$ are representable right fibrations.
  We have to prove that $\R_n$ is representable for $n \geq 2$, which means we have to show it has a final object. 
  By the Segal condition we have a trivial Reedy fibration 
  $$\R_n \to \R_1 \underset{\R_0}{\times} ...\underset{\R_0}{\times} \R_1$$
  Thus it suffices to prove that the right hand side has a final object. However, by the representability condition we have an equivalence
  $$\R_1 \underset{\R_0}{\times} ...\underset{\R_0}{\times} \R_1 \simeq 
  \C_{/W_1} \underset{\C_{/W_0}}{\times} ...\underset{\C_{/W_0}}{\times} \C_{/W_1}$$
  where $W_1$ represents $\R_1$ and $W_0$ represents $\R_0$. But the right hand CSS has a final object if and only if the 
  induced diagram of 
  $$W_1 \rightarrow W_0 \leftarrow W_1 ... W_1 \rightarrow W_0 \leftarrow W_1$$
  has a limit, which holds as $\C$ has finite limits.
 \end{proof}

 \begin{theone} \label{The Segal Obj out of two obj}
  Let $W_0$ and $W_1$ be two objects in a CSS $\C$ with finite limits.
  Let $\R$ be a Segal Cartesian fibration over $\C$ such that $\R_0$ is represented by $W_0$ and $\R_1$ is represented by 
  $W_1$. Then there exists a Segal object $\hat{W}_{\bullet}$ such that $\hat{W}_0 \simeq W_0$ and $\hat{W}_1 \simeq W_1$. 
  Moreover, $\hat{W}$ is complete if $\R$ is a Cartesian fibration.
 \end{theone}

 \begin{proof}
  By Theorem \ref{The Rep Segal Cart} the Segal Cartesian fibration $\R$ is representable.
  Thus, by Theorem \ref{The Building Simp Obj} there exists a simplicial object $\hat{W}$ such that $\hat{W}_0 \simeq W_0$ and
  $\hat{W}_1 \simeq W_1$. Moreover, $\R$ being a Segal Cartesian fibration implies that  $\hat{W}$ is actually a Segal object.
  Finally, $\R$ is a Cartesian fibration if and only if $\hat{W}$ is complete. 
 \end{proof}

 \begin{notone}
  Henceforth we will denote simplicial object $\hat{X}$ with $X$ as well in order to avoid extra notation 
  that can cause extra confusion.
 \end{notone}
 
 

%
%
%

 \begin{remone}
  The idea of this subsection is that we want to built a Segal object in a CSS with finite limits, the same way we build 
  the nerve in ordinary categories, namely by defining the $n$th level to be the pullback 
  $W_1 \times_{W_0} ... \times_{W_0} W_1$. 
  The problem is that in a higher category there are coherence issues which make it very difficult to define 
  all the necessary simplicial maps from between those objects.
  \par
  By using Segal Cartesian fibrations, we can simplify those issues as Segal Cartesian fibrations are fibrant objects in a model structure
  and thus we can get the desired result in an ordinary category.
 \end{remone}

\section{Category Theory of Complete Segal Objects} \label{Sec Category Theory of Complete Segal Objects}
 Now that we have developed a theory of Cartesian fibrations and showed that there is a subclass of 
 Cartesian fibrations completely determined by complete Segal objects, namely representable Cartesian fibrations, 
 we can use this larger framework to develop the category theory of complete Segal objects.
 
 \subsection{The Fundamental Theorem of Complete Segal Objects} \label{Subsec The Fundamental Theorem of Complete Segal Objects}
 One of the very important results of higher category theory is following theorem.
 
 \begin{theone}
  A map of CSS $F:\C_1 \to \C_2$ is an equivalence if and only if it is 
    \begin{enumerate}
     \item Fully Faithful: For any two objects $x,y$ in $\C_1$ the induced map 
     $$map_{\C_1}(x_1,x_2) \to map_{\C_2}(Fx_1,Fx_2)$$
     is an equivalence of spaces.
     \item Essentially Surjective: For any object $y$ in $\C_2$ there exists an object $x$ in $\C_1$ such that 
     $Fx$ is equivalent to $y$.
    \end{enumerate}
 \end{theone}

 This is called the ``Fundamental Theorem of Quasi-Categories" in \cite[Page 71]{Re16}. We want to adopt that result 
 to the setting of complete Segal objects, which we appropriately call the 
 {\it Fundamental Theorem of complete Segal objects}.
 
 \begin{theone}
  Let $\C$ be locally Cartesian closed. 
  A map of CSO $F:W \to V$ in $\C$ is an equivalence if and only if it satisfies following two conditions.
  \begin{enumerate}
   \item Fully Faithful: For two object $x_1,x_2: D \to W_0$ with context $D$ the induced map 
   $$map_W(x_1,x_2) \to map_V(Fx_1,Fx_2)$$
   is an equivalence in $\C$.
   \item Essentially Surjective: For any object $y: D \to V_0$ with context $D$, there exists an object $x: D \to W_0$ with context $D$
   such that $Fx$ is equivalent to $y$ in $V$.
  \end{enumerate}
 \end{theone}
 
 \begin{proof}
  It suffices to prove that the corresponding map $F: \C_{/W} \to \C_{/V}$ is a equivalence of Cartesian fibrations.
  For that it suffices to prove that the map of fibers $F_D: W_D \to V_D$ is an equivalence.
  As the fibers are CSS, it suffices to prove that the map is fully faithful and essentially surjective.
  \par 
  Let $x,y$ be two objects in $W_D$, which are just maps $x,y:D \to W_0$. This gives us following commutative diagram 
  \begin{center}
   \begin{tikzcd}[row sep=0.5in, column sep=0.5in]
    map_{\C}(*, map_W(x,y)) \arrow[r, "\simeq"] \arrow[d, "\simeq"] & map_{W_D}(x,y) \arrow[d] \\
    map_{\C}(*, map_V(Fx,Fy)) \arrow[r, "\simeq"] & map_{V_D} (Fx,Fy)
   \end{tikzcd}
  \end{center}
  The horizontal maps are equivalences by definition (see Remark \ref{Rem CSO to CSS}).
  The left hand vertical map is an equivalence by assumption. This implies that the right hand vertical map is also an equivalence.
  which is exactly what we wanted to show.
  \par 
  For the next part let $y$ be an object in $V_D$, which means it is a map $y: D \to V$. By the assumption of the theorem 
  there exists an $x : D \to W$ such that $Fx: D \to V$ is equivalent to $y$, which means there exists an object $x$ in $W_D$, 
  such that $Fx$ is equivalent to $y$ in $V_D$. Hence, $F_D$ is also essentially surjective.
 \end{proof}

 \begin{remone}
  From the perspective of this proof we can see why it would not have sufficed to consider an object in a complete Segal object 
  $W$ to be a map out of the final object $* \to W_0$ (Definition \ref{Def Obj of Seg Obj}). 
  We do need to consider objects with different contexts to be able to 
  understand equivalences of complete Segal object in terms of the corresponding complete Segal spaces.
 \end{remone}

 \subsection{Adjunction of Complete Segal Objects} \label{Subsec Adjunctions of Complete Segal Objects}
 In this subsection we use the definition of adjunction for compete Segal spaces to define an adjunction 
 of complete Segal objects.
 
 \begin{defone}
  Let $W$ and $V$ be two CSO. An adjunction of CSO is a Cartesian fibration $p:\A \to \C \times F(1)$
  that satisfies the following two conditions:
  \begin{enumerate}
   \item 
  We have following pullback diagram.
  \begin{center}
   \begin{tikzcd}[row sep=0.5in, column sep=0.5in]
    \C_{/W} \arrow[dr, phantom, "\ulcorner", very near start] \arrow[d] \arrow[r] & \A \arrow[d] & \C_{/V} \arrow[l] \arrow[d] 
    \arrow[dl, phantom, "\urcorner", very near start]\\
    \C \arrow[r, "id_{\C} \times 0"] & \C \times F(1) & \C \arrow[l, "id_{\C} \times 1"'] 
   \end{tikzcd}
  \end{center}
  \item 
  For each map $f:F(1) \to \C$  
  \begin{center}
   \pbsq{\A \underset{F(1) \times \C}{\times} F(1)}{\A}{F(1)}{F(1) \times \C}{}{(id \comma f)^*p}{p}{(id \comma f)}
  \end{center}
  the induced map $(id \comma f)^*p$ is a coCartesian fibration.
  \end{enumerate}
 \end{defone}
 
 \begin{remone}
  Note that Cartesian fibrations are stable under pullbacks and so $(id \comma f)^*p$ will also be a Cartesian fibration.
  Thus we could have replaced $(2)$ with the following condition:
  \begin{itemize}
   \item[(2')] For each map $f:F(1) \to \C$  
  \begin{center}
   \pbsq{\A \underset{F(1) \times \C}{\times} F(1)}{\A}{F(1)}{F(1) \times \C}{}{(id \comma f)^*p}{p}{(id \comma f)}
  \end{center}
  the induced map $(id \comma f)^*p$ is an adjunction.
  \end{itemize}
 \end{remone}

 An adjunction of CSO is closely related to adjunction of CSS.
 
 \begin{theone}
  A map $\A \to \C \times F(1)$ is an adjunction of CSO $W$ and $V$ if and only if the map $\A \to \C \times F(1)$ 
  is a CSS fibration and for each object $D$ 
  the map of fibers $\A_D \to F(1)$ is an adjunction of CSS between $W_D$ and $V_D$.
 \end{theone}
 
 \begin{proof}
  Let $p: \A \to \C \times F(1)$ is adjunction of CSO and $D: F(0) \to \C$ an object in $\C$.
  Pulling along $D$ we get 
  $\A_D \to F(1)$, which is by assumption is a coCartesian and Cartesian fibration over $F(1)$. Thus we get an adjunction from 
  $W_D$ to $V_D$.
  \par 
  On the other side, we have to prove that $\A \to \C \times F(1)$ is a Cartesian fibration such that for each object $D$
  the fiber is $\A_D \to F(1)$ is a coCartesian fibration.
  The fact that $\A_D \to F(1)$ is a coCartesian fibration follows from the fact that it is an adjunction of CSS.
  Thus we only have to prove that $\A \to \C \times F(1)$ is a Cartesian fibration.
  \par 
  Let $F(1) \to \C \times F(1)$ be a map , which gives us a map $f:D_1 \to D_2$ in $\C$ and a map $F(1) \to F(1)$.
  We have to show that every such map has a choice of Cartesian lift. There are three cases:
  \begin{enumerate}
   \item The map $F(1) \to F(1)$ is the constant map to $0$. The fiber over $\C \times F(0)$ is just the Cartesian fibration $\C_{/W} \to \C$.
   In this case, any choice of lift of the target to an object $D_2 \to W_0$ has a Cartesian lift $D_1 \to D_2 \to W_0$.
   \item The map $F(1) \to F(1)$ is the constant map to $1$. This case is similar to the previous case as the fiber is just $\C_{/V}$.
   \item The map $F(1) \to F(1)$ is the identity map. In this case the fiber over $F(1)$ is $\A_D \to F(1)$, 
   which is a Cartesian fibration by assumption.
  \end{enumerate}
%
%
%
 \end{proof}

 \subsection{Limits of Complete Segal Objects} \label{Subsec Limits of Complete Segal Objects}
 In this subsection we use representable Cartesian fibrations to define limits of complete Segal objects.
 
 For this subsection let $W$ be a complete Segal object in $\C$ and $I$ be any simplicial object in $\C$.
 
 \begin{defone}
  We define the  $(\C_{/W})^I$ as the internal mapping object
  $$(\C_{/W})^I = (\C_{/W} \to \C)^{(\C_{/I} \to \C)}$$
  Notice the map $(\C_{/W})^I \to \C$ is a Cartesian fibration as the fiber over the object $D$ is equivalent to the CSS $(W_D)^{I_D}$.
 \end{defone}
 
 \begin{defone}
  There is a map 
  $$\Delta_W: \C_{/W} \to (\C_{/W})^I$$
  over $\C$ that is induced by the map $\C_{/I} \to \C$ and the fact that the internal mapping object 
  $(\C_{/W} \to \C)^{(id:_{/C}:\C \to \C)} = \C_{/W}$.
 \end{defone}

 \begin{defone}
  A CSO $W$ has limits of shape $I$ if the map $\Delta_W: \C_{/W} \to (\C_{/W})^I$ has a right adjoint and has a colimits of shape $I$
  if the map has a left adjoint.
 \end{defone}
 
 \begin{exone}
  Let $I = *$ be the final object in $\C$. Then the map $\Delta_W: \C_{/W} \to (\C_{/W})^I$ is just the projection map 
  $\C_{/W} \to \C$. In this case any left adjoint gives us an initial object in $W$ and a right adjoint gives us a
  final object in $W$.
 \end{exone}
 
 \begin{exone}
  For any $n$ let $\Delta_n: W \to W^n$ be the diagonal map in $\C$. This gives us a map $\C_{/W} \to \C_{/W^n}$.
  A left adjoint to this map gives us finite coproducts and the right adjoint gives us finite products.
 \end{exone}
 
 We have following representability theorem for limits and colimits.
 
 \begin{theone} \label{The CSS limits to CSO limits}
 $W$ has (co)limits of shape $I$ if and only if for each object $D$ the CSS $W_D$ has (co)limits of shape $I_D$.
 \end{theone}
 
 \begin{proof}
  The Cartesian fibration $(\C_{/W})^I$ has fiber over $D$ equivalent to $(W_D)^{I_D}$ and the map 
  $\Delta_W: \C_{/W} \to (\C_{/W})^I$ gives us a map $(\Delta_W)_D: W_D \to (W_D)^{I_D}$.
  Now $W$ has (co)limits of shape $I$ if and only if the map $\Delta_W$ has right (left) adjoint, 
  which is equivalent to $(\Delta_W)_D$ having right (left) adjoint. Finally, this is equivalent to $W_D$ having (co)limits
  of shape $I_D$.
 \end{proof} 

 \begin{remone}
  This theorem is telling us that a limit cone in $W$ is a choice of limit cone in $W_D$  that changes functorially 
  with the object $D$. For example if $W$ has a final object $\C \to \C_{/W}$ then for any map $f: D \to D'$ in $\C$ 
  the induced map $F_D: W_D \to W_{D'}$ will preserve the final object.
 \end{remone}
 
 There is a different way to define limits for a complete Segal objects. 
 
 \begin{defone} \label{Def Final Object in CSO}
  An object $f$ in $W$ without context is final if in the following pullback square
  \begin{center}
   \pbsq{(W_{/f})_0}{W_1}{W_0 \times *}{W_0 \times W_0}{}{}{(s \comma t)}{id \times f}
  \end{center}
 the induced map $(W_{/f})_0 \to W_0$ is an equivalence in $\C$.
 \end{defone}
 
 There is following lemma to identify final objects in $W$.
 
 \begin{lemone}
  An object $* \to W$ is final if and only if the induced object in $W_D$ is final for every $D$.
 \end{lemone}

 \begin{proof}
  An object in a CSS $f_D$ in $W_D$ is final if and only if the projection map 
  $$(W_D)_1 \underset{(W_D)_0}{\times} * \to (W_D)_0$$
  is an equivalence, which is exactly the condition we stated in the definition.
 \end{proof}
 
 Based on this lemma we have following definition and corollary.
 
 \begin{defone}
  A Cartesian fibration $p: \R \to \C$ has a final object if there exists a section $\C \to \R$ such that for each object $d$ in $\R$
  the induced map on fibers $* \to * \times_{\C} \R$ is a final object.
 \end{defone}
 
 \begin{corone}
  A CSO $W$ has a final object if and only if $\C_{/W}$ has a final object.
 \end{corone}

%
%

 Now we can use the definition of a final object to define general limits.
 
 \begin{defone} \label{Def Cones in CSO}
  Let $f: I \to W$ be a diagram in $W$. We define the Cartesian fibration of cones as 
  $$\C_{/f} = \C_{/W} \underset{(\C_{/W})^I}{\times} ((\C_{/W})^I)^{F(1)} \underset{(\C_{/W})^I}{\times} \C$$
  where the map $\C \to (\C_{/W})^I$ is induced by the map $f: \C_{/I} \to \C_{/W}$.
 \end{defone}
 
 \begin{defone} \label{Def Limits in CSO}
  The diagram $f: I \to W$ has a limits if the Cartesian fibration $\C_{/f}$ has a final object.
 \end{defone}

 Using Theorem \ref{The CSS limits to CSO limits} we get following corollaries. 
 
 \begin{corone}
  The map $\Delta_W: \C_{/W} \to (\C_{/W})^I$ has a right adjoint if and only if for each $f: I \to W$, the Cartesian fibration 
  $\C_{/f}$ has a final object.
 \end{corone}
 
 \begin{proof}
  $\Delta_W: \C_{/W} \to (\C_{/W})^I$ has a right adjoint if and only if the map $W_D \to (W_D)^{I_D}$ has a right adjoint
  which is equivalent to $(W_D)_{/f_D}$ having a fina object, which finally is equivalent to $(\C_{/W})^I$
  having a final object.
 \end{proof}

\section{Examples of Complete Segal Objects} \label{Sec Examples of Complete Segal Objects}
  In this section we take a closer look at examples of complete Segal objects.
  
  {\bf Classical Categories:}
  Let $C$ be a category with finite limits and $\C = \n C$ be the classification diagram of this category. This gives us a CSS.
  The functor $\n$ is an embedding and so a simplicial object in $\C$ is a {\it functor} 
  from the {\it category} 
  $$S:\Delta^{op} \to C.$$
  \par 
  From this perspective a Segal object in $\C$ is a functor
  $$S: \Delta^{op} \to C$$
  such that for each $n \geq 2$ the induced map 
  $$S_n \xrightarrow{ \ \ \cong \ \ } S_1 \underset{S_0}{\times} ... \underset{S_0}{\times} S_1 $$
  is an isomorphism. This is exactly the data of a category object in $C$ with objects $S_0$ and morphisms $S_1$.
  \par 
  The Segal object $S: \Delta^{op} \to C$ is complete if and only if each isomorphism in this category object is the identity map.
  
  \begin{exone}
   In particular in the category of sets, $\set$, Segal objects are just categories and complete Segal objects 
   are categories without non-trivial automorphisms.
  \end{exone}
  
  Notice in this case all categorical arguments developed for complete Segal objects just translate to the usual argument
  for ordinary categories.
  
 {\bf Spaces:} 
 Let $\spaces$ be a CSS that models spaces.
 In this case a simplicial object is a simplicial space. 
 However, a Segal object is not precisely a Segal space as it might not satisfy the Reedy fibrancy condition.
 We need to find a way to strictify our Segal object in a correct way, to get an actual Segal space.
 \par 
 Every Segal object $W$ gives us a Segal Cartesian fibration $\spaces_{/W}$. Let $\hat{W}$ be the fiber over the final object. 
 This is a Segal space such that there is an equivalence of spaces 
 $$W_n \simeq \hat{W}_n$$
 Thus, every Segal object in spaces is equivalent to a Segal space.
 \par 
  Using the same argument we can show that every complete Segal object in spaces is equivalent to a complete Segal space.
  \par 
  Notice there are some subtle distinctions between complete Segal objects in spaces and complete Segal spaces.
  For example, in a complete Segal object an object is a map of spaces $D \to W_0$, whereas 
  an object in a complete Segal space is defined as a map from the final object $\Delta[0] \to W_0$.
  \par 
  Although the definitions might seem apparently different they are still consistent, as spaces are generated
  by the final object and so it suffices to determine the maps out of the final object. 
  From that perspective defining an object as a map $D \to W_0$ is the correct generalization to a general higher category.
  Thus the theory of complete Segal objects properly generalizes the theory of CSS. 
 
 {\bf Complete Segal Spaces:}
 Let $\CSS$ be the (large) CSS of small CSS. 
 A complete Segal object in $\CSS$ is then exactly a {\it double higher category}.
 This naming convention comes from the fact that category objects in categories are commonly known as {\it double categories}.
 \par 
 Concretely a double higher category is a bisimplicial space $W_{\bullet \bullet}$ such that for each $n$, each of 
 the simplicial spaces $W_{n \bullet}$ and $W_{\bullet n}$ are CSS. 
 \par 
 Notice the categorical constructions from complete Segal objects are compatible with complete Segal objects.
 For example an adjunction of double higher categories $W \to V$ is the data of adjunctions 
 $W_k \to V_k$ for each $k$.
 
 {\bf Right Fibrations:}
 Let $\R Fib(\C)$ be the CSS of right fibrations over a given CSS $\C$.
 In this case a Segal object should correspond to Segal Cartesian fibrations, however as in the case of spaces
 we do have to worry about fibrancy conditions. Thus we use representable Cartesian fibrations. 
 \par 
 Let $\R Fib(\C)_{/W} \to  \R Fib(\C)$ be the Segal Cartesian fibration corresponding to the Segal object $W$ in $\R Fib(\C)$.
 Let $\hat{\R}$ be the fiber of the projection map over the point $id_{\C}: \C \to \C$ (using the fact that the identity map 
 is always a right fibration). This fiber is exactly a Segal Cartesian fibration over $\C$ as it satisfies the Reedy right 
 fibrancy condition as well as the Segal condition.
 \par 
 Using a similar argument we can show that a complete Segal object in $\R Fib(\C)$ is just a Cartesian fibration.

 {\bf Stable CSS:}
 Stable higher categories are used in stable homotopy theory and in particular in the study of spectra.
 
 \begin{defone}
  A CSS is stable if it satisfies following conditions.
  \begin{enumerate}
   \item It has finite limits and colimits.
   \item The initial object and final object are equivalent (it is pointed)
   \item A commutative square is a pullback square if and only if it is a pushout square.
  \end{enumerate}
 \end{defone}
 
 \begin{remone}
  There are many other ways to define stable higher categories. The reason we chose this particular definition will become clear later on.
  For an introduction to stable higher category theory see \cite{Ra11}.
 \end{remone}

 Before we can classify complete Segal objects in a stable higher category, we need following lemma.
 
 \begin{lemone}
  Let $\C$ be a CSS such that every $-1$-truncated map is an equivalence. Then there are no non-trivial complete Segal objects.
 \end{lemone}
 
 \begin{proof}
  Let $W$ be a complete Segal object $\C$. Then it is trivial if and only if the map $W_0 \to W_1$ is an equivalence.
  However, we already know that map is $-1$-truncated and so it is an equivalence.
 \end{proof}
 
 \begin{theone}
  Let $\C$ be a stable CSS. Then every CSO is trivial.
 \end{theone}
 
 \begin{proof}
  Based on the previous lemma it suffices to prove that every ($-1$)-truncated map is an equivalence.
  A map $f: X \to Y$ is a ($-1$)-truncated if and only if the square 
  \begin{center}
   \pbsq{X}{X}{X}{Y}{}{}{f}{f}
  \end{center}
  is a pullback square. However, by the stability condition this is equivalent to the same square being a pushout square.
  However, the pushout along the identity map is necessary an equivalence. 
 \end{proof}

\section{Complete Segal Objects and Univalence} \label{Sec Complete Segal Objects and Univalence}
Univalence is a key concept in homotopy type theory \cite{UF13}, which is a foundational approach to mathematics
that is homotopy invariant. One aspect of homotopy type theory is the construction of models,
which are higher categories that satisfy the constructions and axioms of homotopy type theory.
One of the axioms of homotopy type theory is the {\it univalence axiom}, which then leads to a notion of a 
univalent map in a higher category.
\par 
In this section we show how we can use Segal objects to study univalent maps in locally Cartesian closed higher categories
that are not necessary presentable. I am indebted to Mike Shulman for several helpful conversation and ideas.

 \subsection{History of Univalent Maps} \label{Subsec History of Univalent maps}
 The first model of homotopy type theory was constructed by Kapulkin and Lumsdaine (following Voevodsky)
 using Kan complexes. As part of their work they defined a {\it univalent fibration} \cite[Definition 3.2.10]{KL12},
 which allowed them to prove that the univalence axiom holds in their model. Notice their definition relied on the fact that 
 simplicial sets have a model structure, namely the Kan model structure.
 \par 
 Gepner and Kock generalized this definition of a univalent fibration to the setting of locally Cartesian closed presentable quasi-categories 
 and defined a {\it univalent family} \cite[3.2]{GK17}, removing the model-dependence.
 The goal of this subsection is to give a basic review of the approach that Gepner and Kock used to define univalence
 and how it has to be adjusted in the non-presentable setting.

 \begin{remone}
  We will not give precise definitions as they can already be found in \cite{GK17}, but rather focus on the ideas and give 
  proper references.
 \end{remone}
 
  \begin{remone}
  For this subsection let $\P$ be a fixed quasi-category, in order to be consistent 
  with the model of $(\infty,1)$-categories used in \cite{GK17}.
  However, it should be noted that their definitions do not 
  depend on any particular property of quasi-categories and can be adopted to any other setting.
 \end{remone}
 
 \begin{defone}
  \cite[2]{GK17}
  For two objects $X,Y$ in $\P$, there is a right fibration 
  $$\M ap(X,Y) \to \P$$
  that over the point $T$ in $\P$ has value $map_{\P}(X \times T, Y)$.
 \end{defone}

 We have following representability results for this right fibration.
 
 \begin{propone}
  \cite[Proposition 2.1]{GK17}
  The right fibration $\M ap(X,Y)$ is representable if and only if $\P$ is Cartesian closed.
  We denote the representing object with $\underline{Map}(X,Y)$
 \end{propone}
 
 This right fibration comes with a sub-object defined as follows.
 
 \begin{defone}
  For two objects $X,Y$ in $\P$, there is a right fibration 
  $$\E q(X,Y) \to \P$$
  that over the point $T$ in $\P$ has value $hoequiv_{\P}(X \times T, Y)$.
  It is a subobject of the right fibration $\M ap(X,Y)$.
 \end{defone}
 
 We have following key result about this right fibration.
 
 \begin{propone}
  \cite[Proposition 2.9]{GK17}
  If $\P$ is Cartesian closed and presentable then the right fibration $\E q(X,Y)$ is representable.
  We denote the representing object with $\underline{Eq}(X,Y)$.
 \end{propone}
 
 There results have following obvious corollaries.
 
 \begin{corone}
  \cite[Corollary 2.5]{GK17} \cite[Theorem 2.10]{GK17}
  Let $\P$ be locally Cartesian closed and presentable. Then for any morphisms $f: X \to T$ and $g: Y \to T$
  the right fibrations $\M ap_T(f,g)$ and $\E q_T(f,g)$ are representable right fibrations over $\P_{/T}$,
  with representing objects $\underline{Map}_T(f,g)$ and $\underline{Eq}_T(f,g)$, respectively.
 \end{corone}

 We now have all the ingredients to define univalent families.
 
  Let $\P$ be a locally Cartesian closed presentable quasi-category.
  Moreover, let $p: X \to S$ be a fixed map in $\P$.
 
 \begin{defone}
 \cite[3.1]{GK17}
  Let $\E q_{/S}(X) = \E q_{/ S \times S}(\pi_1^*X,\pi_2^*X)$ as a right fibration and let $\underline{Eq}_{/S}(X)$
  be the representing object in $\P_{/S \times S}$.
 \end{defone}

 \begin{remone}
  Notice there is a map from the identity right fibration $S \to \underline{Eq}_{/S}(X)$, which sends each object $T \to S $ in $\P_{/S}$
  to the identity map in $\E q_{S \times S}(X \times_S T, X \times_S T)$.
 \end{remone}
 
 \begin{defone}
 \cite[3.2]{GK17}
  The map $p:X \to S$ is a univalent family if the map $$\delta_S: S \to \underline{Eq}_{/S}(X)$$ is an equivalence.
 \end{defone}
 
 Now what goes wrong if the we don't have the presentability condition? 
 We can define the object $\underline{Map}_{/T}(f,g)$ as that only require the category to be locally Cartesian closed.
 However, we cannot define $\underline{Eq}_{/T}(f,g)$ as we only know it exists because of presentability.
 More precisely, we are trying to build a specific subobject of  $\underline{Map}_{/T}(f,g)$ that represents the equivalences.
 However, there is no method to build subobjects in a higher category without presentability.
 \par 
 The way to adjust this is to take the same approach we took when we wanted to define $W_{hoequiv}$ for a given Segal object $W$.
 Instead of using subobjects we used pullbacks to define it (Subsection \ref{Subsec Homotopy Equivalences in Segal Objects}).
 Thus the goal is to construct a Segal object out of the given map and then using the completeness 
 condition to define univalence.
 
 \subsection{Univalence via Completeness} \label{Subsec Univalence vs Completeness}
 The goal of this subsection is to show how we can use complete Segal objects to define univalent maps in a locally Cartesian closed 
 higher category. 
 The definition is quite technical and requires the following steps:
 \begin{enumerate}
  \item Defining a Cartesian fibration for each map.
  \item Modifying the Cartesian fibration to a Segal Cartesian fibration.
  \item Showing that the Segal Cartesian fibration came from a Segal object.
  \item Using the Segal object to define univalent maps.
 \end{enumerate}

 
 \begin{remone}
  For this subsection let $\C$ be a locally Cartesian closed CSS and let $p: E \to B$ be a fixed map in $\C$.
 \end{remone}

 {\bf Defining a Cartesian fibration}:
 In the first part we need to define an important Cartesian fibration out of $p$.
 
%

 \begin{defone}
  \cite[Definition 7.68]{Ra17b}
  Let $t: \O_{\C} \to \C$ be the right fibration over $\C$ which has objects morphisms in $\C$ and has morphisms pullback squares in $\C$.
 \end{defone}
%
%
%
 The right fibration gives us following diagram 
 
 \begin{center}
  \comsq{(\O_\C)_{/p}}{\O_\C}{\C_{/B}}{\C}{\pi_p}{}{}{}
 \end{center}
 
 We use this commutative square to give following definition.
 
 \begin{defone}
  Let $\O^{(p)}_\C$ be the full subcategory of $\O_\C$ generated by all objects in the image of $\pi_p$.
 \end{defone}
 
 We call $\O^{(p)}_{\C}$ the sub category of morphisms generated by pullbacks of $p: E \to B$.
 

 \begin{lemone}
  The map $t: \O^{(p)}_{\C} \to \C$ is a right fibration.
 \end{lemone}

 \begin{proof}
  We have following diagram 
  \begin{center}
   \begin{tikzcd}[row sep=0.5in, column sep=0.5in]
    (\O^{(p)}_{\C})_1 \arrow[r] \arrow[dr, bend right = 20, hookrightarrow] & 
    Pb \arrow[dr, phantom, "\ulcorner", very near start] \arrow[d, hookrightarrow] \arrow[r] & 
    (\O^{(p)}_{\C})_0 \arrow[d, hookrightarrow] \\
     & (\O_{\C})_1 \arrow[r, "t"] & (\O_{\C})_0
   \end{tikzcd}
  \end{center}
  Thus it suffices to show that these two are the same subspace of $(\O_{\C})_1$.
  First, notice that  $(\O_{\C})_1$ is the space of all pullback squares in $\C$.
  \begin{center}
   \begin{tikzcd}[row sep=0.5in, column sep=0.5in]
    A \arrow[r] \arrow[d, "f"] & B \arrow[d, "g"] \\
    X \arrow[r] & Y 
   \end{tikzcd}
  \end{center}
  $Pb$ is the subspace of $(\O_{\C})_1$ generated by all pullback squares such that $g$ is a pullback of $p$.
  On the other hand $(\O^{(p)})_1$ is the subspace of $\O_1$ generated by all points such that both $f$ and $g$ are pullbacks of $p$.
  \par 
  However the composition of two pullback squares again gives us a pullback square. 
  Thus $g$ is a pullback of $p$ if and only if $f$ and $g$ are pullback of $p$. 
  Hence, the they are both the same subspace of $(\O_{\C})_1$, which means that $\O_{\C}^{(p)}$ is a right fibration over $\C$.
 \end{proof}

 Having defined $\O_{\C}^{(p)}$ we can finally finish the first part and define our desired Cartesian fibration.
 
 \begin{defone}
  Let $(\C^{F(1)})^{(p)}$ be the sub-Cartesian fibration of $\C^{F(1)}$ generated by objects $\O_{\C}^{(p)}$. 
  Concretely, 
  $$((\C^{F(1)})^{(p)})_k = (\C^{F(1)})_k \underset{(\O_{\C})^{k+1}}{\times} (\O_{\C}^{(p)})^{k+1}$$
 Notice there is a natural inclusion map $(\C^{F(1)})^{(p)} \hookrightarrow \C^{F(1)}$.
 \end{defone}

 \begin{notone}
  For each object $D$ we denote the fiber of this Cartesian fibration by $(\C_{/D})^{(p)}$.
  We chose this notation as the fiber is the full subcategory of the over-category $\C_{/D}$ 
  generated by maps that can be obtained as a pullback of $p$.
  Using the same argument we denote the fiber of $\O_\C^{(p)}$ over $D$ as $(\O_{/D})^{(p)}$.
 \end{notone}

 {\bf Defining a Segal Cartesian Fibration}:
 The right fibration $\O_{\C}^{(p)}$ has a distinguished point, namely the map $E \to B$ 
 that lies over $B$ in $\C$.
 By the Yoneda Lemma, this point induces a map of right fibrations 
 $$s'_0 = (-)^*p: \C_{/B} \to \O_{\C}^{(p)}.$$
 Concretely, the map takes a morphism $f: D \to B$ to the pullback $f^*(p): f^*E \to D$.
 
 For the next part we first need a technical lemma. 
 
 \begin{lemone} \label{Lemma Arrow level one rep}
  The right fibration $((\C^{F(1)})^{(p)})_1$ is representable.
 \end{lemone}

 \begin{proof}
   Let $\M = (p \times id)_*(E \times E)$ in $\C_{/B \times B}$. Alternatively we can define $\M$ 
 as the internal mapping object
 $$\M =  (B \times E \to B \times B)^{(E \times B \to B \times B)}.$$
 As it lives in $\C_{/B \times B}$ it comes with map $(s,t): \M \to B \times B$.
 The goal is to show that $((\C^{F(1)})^{(p)})_1$ is equivalent to $\C_{/ \M}$.  
 For that we first show that 
 $map_{\C}(D,\M)$ is equivalent to the fiber of $((\C^{F(1)})^{(p)})_1$ over $D$ for every object $D$.
 We will denote this fiber by $((\O_{/D})^{(p)})^{F(1)}$. 
 
 Note that we have diagrams
 
 \begin{center}
  \begin{tikzcd}[row sep=0.5in, column sep=0.5in]
   map_{\C}(D,\M) \arrow[dr, twoheadrightarrow] & & ((\O_{/D})^{(p)})^{F(1)} \arrow[dl, twoheadrightarrow] \\
    & map_{\C}(D,B) \times map_{\C}(D,B) &
  \end{tikzcd}
 \end{center}

 Thus it suffices to compare these spaces fiber-wise. Fix a point 
 $$(f,g): \Delta[0] \to map_{\C}(D,B) \times map_{\C}(D,B).$$
 The fiber of $((\O_{/D})^{(p)})^{F(1)}$ over $(f,g)$ is the space $map_{/D}(f^*E,g^*E)$. 
 We will show the other space has the same fiber.
 \par 
 By adjunction property, the fiber over $(f,g)$ is equivalent to 
 $$map_{/B \times B}( D \underset{(B \times B)}{\times} (E \times B), B \times E)$$
 First we understand the pullback $D \times_{(B \times B)} (E \times B)$.
 It is the limit of the diagram 
 \begin{center}
  \begin{tikzcd}[row sep=0.5in, column sep=0.5in]
   E \arrow[dr, "p"] & & D \arrow[dl, "f"] \arrow[dr, "g"] & & B \arrow[dl, "id_B"] \\
    & B & & B &
  \end{tikzcd}
 \end{center}
 Taking the pullback of this diagram we get following diagram
 \begin{center}
  \begin{tikzcd}[row sep=0.5in, column sep=0.5in]
   & & f^*E \arrow[dl, "id_{f^*E}"] \arrow[dr, "f^*p"] & & \\
   & f^*E \arrow[dl] \arrow[dr, "f^*p"] & & D \arrow[dl, "id_D"] \arrow[dr, "g"] & \\
   E \arrow[dr, "p"] & & D \arrow[dl, "f"] \arrow[dr, "g"] & & B \arrow[dl, "id_B"] \\
    & B & & B &
  \end{tikzcd}
 \end{center}
 
 Notice from the limit diagram above we can tell that the map $f^*E \to B \times B$ are the maps $(f \circ (f^*p), g \circ (f^*p))$.
 Next we know that 
 $$map_{/B \times B}( D \underset{(B \times B)}{\times} (E \times B), B \times E) = map_{/B}(f^*E,B) \times map_{/B}(f^*E,E) 
 \simeq map_{/B}(f^*E,E)$$
 as the space $map_{/B}(f^*E,B)$ is contractible.
 Finally by adjunction we have 
 $$map_{/B}(f^*E,E) \simeq map_{/D}(f^*E,g^*E).$$
 Putting this all together we have following chain of equivalences.
 \begin{center}
  \begin{tikzcd}[row sep=0.5in, column sep=0.5in]
   map_{/B \times B}(D, \M) &  map_{/B \times B}(f^*E, B \times E) \arrow[l, "(p \times id)_*"' ,"\simeq"] 
   \arrow[r, "\pi_1","\simeq"'] & 
 map_{/B}(f^*E,E) \arrow[r, "g^*", "\simeq"'] & map_{/D}(f^*E,g^*E)
  \end{tikzcd}
 \end{center}

 In particular, if $D = \M$, then we have a distinguished point on the left hand side ($id_{\M}: \M \to \M$)
 that gives us a point in $map_{/\M}(s^*E,t^*E)$. We will denote this map by $F_{\M} :F(0) \to ((\C^{F(1)})^{(p)})_1$. 
 This induces a map 
 $$ (F_{\M})^*: \C_{/\M} \to ((\C^{F(1)})^{(p)})_1.$$
 We will show this map is a Reedy equivalence. As both maps are right fibrations over $\C$ it suffices to prove the equivalence fiberwise.
 So we have to show that the map is an equivalence
 $$(F_{\M})^*_D : map_{\C}(D,\M) \to ((\O_{/D})^{(p)})^{F(1)}$$
 Again both are Kan fibrations over $map(D,B) \times map(D,B)$ and so we can check whether the map is an equivalence fiberwise.
 Thus we can fix maps $f,g: D \to B$ and now we have to prove the map 
 $$(F_{\M})^*_{(f,g)} : map_{/B \times B}(D,\M) \to map_{/D}(f^*E,g^*E)$$
 is an equivalence. Notice the map takes a map $h: D \to \M$ to $h^*s^*E \to h^*t^*E$ over $D$ (where we are using the fact that $sh =f$ and $th = g$.
 However, the map fits into following commutative diagram:
 \begin{center}
  \begin{tikzcd}[row sep=0.5in, column sep=0.5in]
    & map_{/B \times B}(f^*E, B \times E) \arrow[dl, "(p \times id)_*"', "\simeq"] \arrow[dr, "g^* \pi_1", "\simeq"'] & \\
    map_{/B \times B}(D, \M) \arrow[rr, "(F_{\M})^*_{(f,g)}"] & & map_{/D}(f^*E,g^*E)
  \end{tikzcd}
 \end{center}
 Thus by two out of three the bottom map is an equivalence of spaces.
 This is exactly the statement we wanted to show and hence we are done.
 \end{proof}

 This lemma plays an important role in the next part. Right now we need following corollary.
 
 \begin{corone}
  There exists a map 
  $$((\C^{F(1)})^{(p)})_1 \to \C_{/ B \times B}$$
 \end{corone}
 
 \begin{proof}
  By the lemma above there exists an equivalence 
  $$\C_{/ \M} \xrightarrow{ \ \ \simeq \ \ } ((\C^{F(1)})^{(p)})_1 .$$
  Moreover, by definition of $\M$ there exists a map $\M \to B \times B$, which by the Yoneda lemma 
  induces a map of right fibrations 
  $$ \C_{/ \M} \to \C_{/ B \times B}$$
  By the previous equivalence we can lift this equivalence to a map 
  \begin{center}
   \begin{tikzcd}[row sep=0.5in, column sep=0.5in]
    \C_{/ \M} \arrow[r] \arrow[d, "\simeq"] & \C_{/ B \times B} \\
    ((\C^{F(1)})^{(p)})_1 \arrow[ur, dashed] & 
   \end{tikzcd}
  \end{center}

 \end{proof}

 We can precompose the map above with the map $((\C^{F(1)})^{(p)})_0 \to ((\C^{F(1)})^{(p)})_1$ to get the map 
 $$(s', t' ): \O_\C^{(p)} = ((\C^{F(1)})^{(p)})_0 \to \C_{/ B \times B}.$$

 Combining the two maps we get following commutative diagram.
 
 \begin{center}
  \begin{tikzcd}[row sep = 0.5in, column sep =0.5in]
   & \O_\C^{(p)} \arrow[d, "(s' \comma t' )"] \\
   \C_{/B} \arrow[r, "d"] \arrow[ur, "s_0'"] & \C_{/B} \times \C_{/ B}
  \end{tikzcd}
 \end{center}
 The goal is to show that the map $d: \C_{/B} \to \C_{/B} \times \C_{/ B}$ is induced by the diagonal map 
 $\Delta_B: B \to B \times B$.
 \par 
 Using the information of the previous lemma we can expand this diagram to the diagram of the following form.
 \begin{center}
  \begin{tikzcd}[row sep = 0.5in, column sep =0.5in]
                                        & \O_\C^{(p)} \arrow[d, "(s' \comma t' )"] \arrow [r] & (\C^{F(1)})_1^{(p)} \arrow[dl] \\
   \C_{/B} \arrow[r, "d"] \arrow[ur, "s_0'"] & \C_{/B} \times \C_{/ B} & \C_{/ \M} \arrow[u, "\simeq"] \arrow[l]
  \end{tikzcd}
 \end{center}
 As the maps above are right fibrations and $\C_{/B}$ is representable it suffices to check the image of the map $id_B: B \to B$ in $\C$. 
 So, first of all we take the fiber of the right fibrations over $B$ which gives us a diagram of spaces.
 
 \begin{center}
  \begin{tikzcd}[row sep = 0.5in, column sep =0.5in]
   & (\O_{/B})^{(p)} \arrow[d, "(s'_B \comma t'_B )"] \arrow[r] & (\C_{/B})^{(p)}_1 \arrow[dl] \\
   map(B,B) \arrow[r, "d_B"] \arrow[ur, "(s_0)_B'"] & map(B,B) \times map(B,B) & map(B,\M) \arrow[u, "\simeq"] \arrow[l]
  \end{tikzcd}
 \end{center}
 The goal now is to trace the image of the identity map through the various maps until we end up in $map(B,B) \times map(B,B)$.
 The image of $id_B$ in $(\O_{/B})^{(p)}$ is just the map $p:E \to B$.
 The inclusion map to $(\C_{/B})^{(p)}_1$ takes it to the commutative triangle
 \begin{center}
  \begin{tikzcd}[row sep=0.5in, column sep =0.5in]
   E \arrow[rr] \arrow[dr] & & E \arrow[dl] \\
    & B & 
  \end{tikzcd}
 \end{center}
 The map $(\C_{/B})^{(p)}_1 \to map(B,B) \times map(B,B)$ then takes the commutative triangle above
 first to the source and target $(p: E \to B, p: E \to B)$ and then to the maps that created them as pullbacks
 $(id_B: B \to B, id_B : B \to B)$. 
 Thus the map $d: B \to B \times B$ is really $\Delta_B: B \to B \times B$. 
 \par 
 Using this new information our previous diagram now turns into the following.
 \begin{center}
  \begin{tikzcd}[row sep = 0.5in, column sep =0.5in]
   & \O_\C^{(p)} \arrow[d, "(s' \comma t' )"] \\
   \C_{/B} \arrow[r, "(\Delta_B)^*"] \arrow[ur, "s_0'"] & \C_{/B} \times \C_{/ B}
  \end{tikzcd}
 \end{center}
  \par 
 Combining all the information we gathered until now gives us an ``extended" simplicial diagram.
 \begin{center}
 \begin{tikzcd}[row sep=0.5in, column sep=0.5in]
   \C_{/B} \arrow[r, shorten >=1ex,shorten <=1ex, "s_0'" near start]
   &  ((\C^{F(1)})^{(p)})_0  
   \arrow[l, shift left=1.2, "s'" near start] \arrow[l, shift right=1.2, "t'"' near start]  
   \arrow[r, shorten >=1ex,shorten <=1ex, "s_0" near start]
   &  ((\C^{F(1)})^{(p)})_1
   \arrow[l, shift left=1.2, "s" near start] \arrow[l, shift right=1.2, "t"' near start] 
   \arrow[r, shift right, shorten >=1ex,shorten <=1ex ] \arrow[r, shift left, shorten >=1ex,shorten <=1ex] 
   &  ((\C^{F(1)})^{(p)})_2
   \arrow[l] \arrow[l, shift left=2, "d_2"] \arrow[l, shift right=2, "d_0"'] 
   \arrow[r, shorten >=1ex,shorten <=1ex] \arrow[r, shift left=2, shorten >=1ex,shorten <=1ex] \arrow[r, shift right=2, shorten >=1ex,shorten <=1ex]
   & \cdots 
   \arrow[l, shift right=1] \arrow[l, shift left=1] \arrow[l, shift right=3] \arrow[l, shift left=3] 
 \end{tikzcd}
 \end{center}
 Using the fact that $t's_0' = id_B$ and $s's_0' = id_B$ we can compose those arrows to get an actual simplicial object.
 \begin{center}
 \begin{tikzcd}[row sep=0.5in, column sep=1in]
   \C_{/B} \arrow[r, shorten >=1ex,shorten <=1ex, "\sigma_0 = s_0s_0'" near start]
   & ((\C^{F(1)})^{(p)})_1
   \arrow[l, shift left=1.2, "\sigma = s's" near start] \arrow[l, shift right=1.2, " \tau = t't"' near start] 
   \arrow[r, shift right, shorten >=1ex,shorten <=1ex ] \arrow[r, shift left, shorten >=1ex,shorten <=1ex] 
   & ((\C^{F(1)})^{(p)})_2 
   \arrow[l] \arrow[l, shift left=2, "d_2"] \arrow[l, shift right=2, "d_0"'] 
   \arrow[r, shorten >=1ex,shorten <=1ex] \arrow[r, shift left=2, shorten >=1ex,shorten <=1ex] \arrow[r, shift right=2, shorten >=1ex,shorten <=1ex]
   & \cdots 
   \arrow[l, shift right=1] \arrow[l, shift left=1] \arrow[l, shift right=3] \arrow[l, shift left=3] 
 \end{tikzcd}
 \end{center}
%
 We will show it is a Segal Cartesian fibration. First of all it is clearly a Reedy right fibration as it is a level-wise right fibration.
 Thus we only have to prove it satisfies the Segal condition.
 For that we need to find the pullback of the diagram 
 \begin{center}
  \begin{tikzcd}[row sep=0.5in, column sep=0.5in]
   & (\C^{F(1)})_1^{(p)} \arrow[d, "\tau"]\\
   (\C^{F(1)})_1^{(p)} \arrow[r, "\sigma"] & \C_{/B}
  \end{tikzcd}
 \end{center}
 We can extend the diagram above to the diagram
 \begin{center}
  \begin{tikzcd}[row sep=0.5in, column sep=0.5in]
   & (\C^{F(1)})_1^{(p)} \arrow[d, "\tau"] \arrow[ddr, bend left = 20, "t"] \\
   (\C^{F(1)})_1^{(p)} \arrow[r, "\sigma"] \arrow[drr, bend right = 20, "s"] & \C_{/B} \arrow[dr, hookrightarrow, "s_0'"] \\
   & & (\C^{F(1)})_0^{(p)}
  \end{tikzcd}
 \end{center}
 The map $s_0'$ is an injection and thus both squares have the same pullback.
 However, $(\C^{F(1)})^{(p)}$ is a Cartesian fibration and thus satisfies the Segal condition. 
 So, both squares are pullback squares:
 \begin{center}
  \begin{tikzcd}[row sep=0.5in, column sep=0.5in]
   (\C^{F(1)})_2^{(p)} \arrow[d] \arrow[r] \arrow[dr, phantom, "\ulcorner", very near start] & 
   (\C^{F(1)})_1^{(p)} \arrow[d, "\tau"] \arrow[ddr, bend left = 20, "t"] \\
   (\C^{F(1)})_1^{(p)} \arrow[r, "\sigma"] \arrow[drr, bend right = 20, "s"] & \C_{/B} \arrow[dr, hookrightarrow, "s_0'"] \\
   & & (\C^{F(1)})_0^{(p)}
  \end{tikzcd}
 \end{center}
 We can use a similar argument to show that the map 
 $$(\C^{F(1)})_k^{(p)} \xrightarrow{ \ \ \simeq \ \ } (\C^{F(1)})_1^{(p)}  \underset{\C_{/B}}{\times} ... 
 \underset{\C_{/B}}{\times} (\C^{F(1)})_1^{(p)}$$
 Having dones all this work we finally give following definition.

 \begin{defone}
  We call the Segal Cartesian fibration described above $(\C^{F(1)})^{(p,B)}$.
 \end{defone}

 We are now ready to move on to the next step.
 
 {\bf The Segal Cartesian fibration comes from a Segal object}:
 In this part we want to show that the Segal Cartesian fibration from the previous part is representable.
 
 \begin{lemone}
  The Segal Cartesian fibration $(\C^{F(1)})^{(p,B)}$ is representable.
 \end{lemone}
 
 \begin{proof}
  By Theorem \ref{The Segal Obj out of two obj} it suffices to prove that $((\C^{F(1)})^{(p,B)})_0$ and $((\C^{F(1)})^{(p,B)})_1$ 
  are representable right fibrations. $((\C^{F(1)})^{(p,B)})_0 = \C_{/B}$ and so is representable by definition.
  And $((\C^{F(1)})^{(p,B)})_1$ is representable by Lemma \ref{Lemma Arrow level one rep}.
  Hence we are done.
 \end{proof}
 
 This lemma gives us following definition.
 
 \begin{defone}
  We denote any choice of Segal object that represents $(\C^{F(1)})^{(p)}$ as $\n(p)$.
 \end{defone}

 \begin{remone}
  The proof above shows that $\n(p)_0 \simeq B$, $\n(p)_1 \simeq \M $.
  Moreover, $\n(p)_n$ is equivalent to the limit $\M \times_B ... \times_B \M$.
 \end{remone}
 
 \begin{remone}
  The notation $\n (p)$ is chosen to remind the reader that the construction is a generalized nerve construction.
 \end{remone}

 We have gathered all the ingredients to move on to the final part.
 
 
 {\bf Defining Univalent Maps}:
 We are finally in a position to give a definition of univalence.

 \begin{defone} \label{Def Univalent map}
  A map $p: E \to B$ is univalent if the Segal object $\n(p)$ is complete.
 \end{defone}
 

 Let us see two basic examples.
 
 \begin{exone}
  Let $id_B: B \to B$ be an identity map. Then $\n(p)_n$ is equivalent to $B^n$, the $n$-fold product of $B$. 
  In this case, $\n(p)_{hoequiv}$ is $B \times B$ and so $id_B$ is univalent if and only if the map $\Delta: B \to B \times B$.
  is an equivalence. This by definition means $B$ is a $(-1)$-truncated object. 
  Hence, $id_B$ is univalent if and only if $B$ is $(-1)$-truncated.
  In particular, the identity map of the final object $* \to *$ is univalent.
 \end{exone}
 
 \begin{exone} \label{Ex Uni over point}
  Let $F \to *$ be the map to the final object. In this case $\n(p)_1 = F^F$, the mapping object of maps from $F$ to itself.
  Moreover, $map(*,\n(p)_{hoequiv}) \simeq hoequiv_{\C}(F,F)$ and so $F \to *$ is univalent if and only if the space of self-equivalences
  of $F$ is contractible. 
 \end{exone}


 There are several equivalent ways to define univalence.
 First here is a basic, but valuable lemma.
 \begin{lemone}
  There is an equivalence of right fibrations.
  $$ (\C_{/\n(p)})_{hoequiv} \simeq \O^{(p)}_{\C}$$
 \end{lemone}

 \begin{theone}
  The following are equivalent.
  \begin{enumerate}
   \item $p$ is univalent.
   \item The map $\C_{/B} \to \O^{(p)}_{\C}$ is an equivalence.
   \item The map $E \to B$ is a final object in the CSS $\O^{(p)}_{\C}$.
   \item The Cartesian fibration $(\C^{F(1)})^{(p)}$ is representable.
   \item The map $(\C^{F(1)})^{(p,B)} \to \C^{F(1)}$ is $(-1)$-truncated.
  \end{enumerate}
 \end{theone}

 \begin{proof}
  We go through the different cases:
  
  ${ ( 1 \Longleftrightarrow 2)}$
  $p$ is univalent if and only if $\n(p)_0=B$ is equivalent to $\n(p)_{hoequiv}$, which is equivalent to 
  $\C_{/B}$ being equivalent to $\C_{hoequiv}$. By the previous lemma this is the same as $\C_{/B}$ being equivalent to $\O^{(p)}_{\C}$.
  
  ${ (2 \Longleftrightarrow 3)}$ $\C_{/B}$ has a final object by definition. Thus $\O^{(p)}_{\C}$ has final object $E \to B$ if and only if 
  it is equivalent to $\C_{/B}$.
  
  ${ (2 \Longleftrightarrow 4)}$ The Cartesian fibration $(\C^{F(1)})^{(p)}$ is representable if and only if $\O^{(p)}_{\C}$ is representable.
  
  ${ (2 \Longrightarrow 5)}$ We have the chain of maps $(\C^{F(1)})^{(p,B)} \to  (\C^{F(1)})^{(p)} \to \C^{F(1)}$, 
  where the second map is $-1$-truncated.
  If the map $\C_{/B} \to \O^{(p)}_{\C}$ is an equivalence then the composition is $-1$-truncated. 
  \par 
  If the map $(\C^{F(1)})^{(p,B)} \to \C^{F(1)}$ is $(-1)$-truncated then in particular the map 
  $(\C^{F(1)})^{(p,B)} \to (\C^{F(1)})^{(p)}$ is $-1$-truncated, which also implies that the map $\C_{/B} \to \O^{(p)}_{\C}$ is $-1$-truncated.
  We will show this map is an equivalence. As it is a map of right fibrations it suffices to do so fiber-wise.
  Thus we have to show that for each object $map(D,B) \to (\C_{/D})^{(p)}$ is an equivalence.
  We already know it is $-1$-truncated, but we also know it is surjective on path-components. Thus it has to be an equivalence.
 \end{proof}

 \subsection{The Poset of Univalent Maps} \label{Subsec The Poset of Univalent Maps}
 In this subsection we want to discuss how univalent maps relate to each other.
 For this subsection we fix a pullback square
 \begin{center}
  \pbsq{E_2}{E_1}{B_2}{B_1}{f_E}{p_2}{p_1}{f_B}
 \end{center}
 in a locally Cartesian closed higher category $\C$.
 
 First of all the pullback induces an embedding of Cartesian fibrations 
 $$(\C^{F(1)})^{(p_2)} \to (\C^{F(1)})^{(p_1)}$$
 as every map that is a pullback of $p_2$ is also a pullback of $p_1$.
 \par 
 This map induces an embedding of Segal Cartesian fibrations 
 $$(\C^{F(1)})^{(p_2, B_2)} \to (\C^{F(1)})^{(p_1, B_1)}$$
 which by the Yoneda lemma gives us a map of Segal objects 
 $$\n(p_2) \to \n(p_1)$$
 In light of this map we have following theorem.
 
 \begin{theone} \label{The Uni iff Mono}
  Assume $p_1$ is univalent. Then $p_2$ is univalent if and only if $f_B$ is mono.
 \end{theone}
 
 \begin{proof}
  The map of $p_2$ is univalent if and only if $\n(p_2)$ is a complete Segal object.
  This is equivalent to $(\C^{F(1)})^{(p_2, B_2)}$ being a Cartesian fibration, which is equivalent to being a fiberwise CSS.
  However, as we have a pullback square the induced map of CSS is always an embedding.
  Thus the proof reduces to proving the following statement:
  \par 
  {\it Let  $F:\C_2 \to \C_1$ be an embedding from a Segal space $\C_2$ to complete Segal space $\C_1$. Then 
  $\C_2$ is a complete Segal space if and only if the map of spaces $(\C_2)_0 \to (\C_1)_0$ is $(-1)$-truncated.}
  \par 
  We prove this statement in the following way. Let $x,y$ be two objects in $\C_2$. This gives us following commutative diagram
  \begin{center}
   \comsq{hoequiv_{\C_2}(x,y)}{hoequiv_{\C_1}(Fx,Fy)}{Path_{(\C_2)_0}(x,y)}{Path_{(\C_1)_0}(Fx,Fy)}{\simeq}{}{\simeq}{}
  \end{center}
  $\C_2$ is complete if and only if the right hand vertical map is an equivalence. This is equivalent to the map of path spaces 
  $Path_{(\C_2)_0}(x,y) \to Path_{(\C_1)_0}(Fx,Fy)$, being an equivalence. However, this is just the statement 
  that the map $(\C_2)_0 \to (\C_1)_0$ is a $(-1)$-truncated map of spaces.
  
 \end{proof}
 
 This theorem guides us towards our understanding of univalent maps.
 
 \begin{theone} \label{The Univalent Poset}
  The sub-category of $\O_{\C}$ generated by all univalent maps is a poset.
 \end{theone}
 
 \begin{proof}
  This follows from the fact that any univalent map $p: E \to B$ is final in $\O^{(p)}_{\C}$ and so 
  the space of maps between two univalent maps is either empty or if not then has to be contractible.
 \end{proof}

 \begin{remone}
  Both of those Theorems were proven in the presentable case in \cite[Theorem 3.9, Corollary 3.10]{GK17}
 \end{remone}

 \begin{remone}
  If $\C$ is presentable then we can define bounded local classes of maps \cite[3.3]{GK17}.
  Intuitively, it is a subclass of the morphisms in $\C$ that are closed under base change, satisfies a certain locality condition 
  (sheaf condition) and has some cardinality bound. Those bounded local classes form a poset under inclusion.
  \cite[Theorem 3.9]{GK17} shows that in a presentable locally Cartesian closed quasi-category there is an equivalence between the poset of bounded
  local classes and the poset of univalent families. 
  \par 
  In an arbitrary locally Cartesian closed higher category we cannot define bounded local classes, but we can still define univalent maps.
  Thus in the non-presentable setting univalence takes the role of bounded local classes. In other words, we can use univalence to define
  bounded local classes.
 \end{remone}

 \subsection{Univalence and Elementary Toposes} \label{Subsec Univalence and Elementary Toposes}
 In classical category theory there is a class of categories known as {\it elementary toposes},
 which is used in categorical logic. Here we only give some basic definitions necessary to study univalent maps.
 For a detailed introduction see \cite{MM92}.
 
 \begin{remone}
  In this subsection we are completely focusing on classical categories ($1$-categories).
 \end{remone}

 \begin{defone}
  Let $\C$ be a (classical) category with finite limits. There is a functor 
  $$Sub(-): \C^{op} \to \set$$
  that takes each object $c$ to the set of equivalence classes of subobjects of $c$ (mono maps with target $c$), which we denote by $Sub(c)$.
  The functoriality follows from the fact that the pullback of a mono map is also mono.
 \end{defone}

 \begin{defone}
  Let $\C$ be a (classical) category with finite limits.  An object $\Omega$ is called a {\it subobject classifier}
  if it represents $Sub(-)$.
 \end{defone}

 \begin{remone}
  If $\Omega$ is a subobject classifier then we can deduce from the  equivalence $Hom(\Omega, \Omega) \cong Sub(\Omega)$ 
  the existence of a {\it universal mono} $u:1 \to \Omega$ such that for every mono map $A \to B$ there exists a
  pullback square 
  \begin{center}
   \pbsq{A}{1}{B}{\Omega}{}{}{u}{}
  \end{center}
  Here $1$ is the final object in $\C$. 
 \end{remone}

 \begin{defone}
  A category $\E$ is an elementary topos if it is locally Cartesian closed and has a subobject classifier $\Omega$.
 \end{defone}
 
 We have following results about univalent maps in an elementary topos.
 
 \begin{propone}
  The universal map $u: 1 \to \Omega$ is univalent.
 \end{propone}
 
 \begin{proof}
  The category object $\n(u)$ has objects $\Omega$ and the morphisms come from the Heyting algebra structure on $\Omega$. 
  This is exactly the internal Heyting object as described in \cite[Theorem 1, Page 201]{MM92}, which in particular is a poset.
  But a poset never has non-trivial automorphism and thus is a complete Segal object in $\E$.
 \end{proof}

 \begin{propone}
  A mono map $v: E \to B$ is univalent if and only if $B$ is a subobject of $\Omega$.
 \end{propone}

 \begin{proof}
  As the map $v: E \to B$ is mono there is a pullback square of the form
  \begin{center}
   \pbsq{E}{1}{B}{\Omega}{}{v}{u}{i}
  \end{center}
  By Theorem \ref{The Uni iff Mono} $v$ is univalent if and only if $i$ is mono.
 \end{proof}
  
 \begin{exone}
  The category $\set$ is an elementary topos where $\Omega = \{ 0,1\}$. Thus the poset of mono univalent maps has the $4$ objects
  
  \begin{center}
  \begin{tabular}{ccccccc}
   $\emptyset$ & & $\emptyset$ & & $\{ 1 \}$ & & $\{ 1 \}$ \\
   $\downarrow$ & , & $\downarrow$ & , & $\downarrow$ & , & $\downarrow$ \\
   $\emptyset$ & & $\{ 1 \}$ & & $\{ 1 \}$ & & $\{ 0, 1 \}$
  \end{tabular}
  \end{center}
 \end{exone}

 The classification above actually recovers all univalent maps in $\set$.
 
 \begin{lemone}
  If a map in $\set$ is univalent then it has to be mono.
 \end{lemone}

 \begin{proof}
  If the map $p: E \to B$ is univalent in $\set$ then for every element $b \in B$ then fiber map $E_b \to \{ b \}$ also has to be 
  univalent as it is the pullback along the mono map $b: * \to B$. However, by Example \ref{Ex Uni over point} this only happens if
  $E_b$ has no non-trivial automorphisms. In the category of sets this only holds if $E_b$ is either empty or has one object.
  We just showed that the fiber over each point $b$ is either empty or has one point and so $E \to B$ is an injection of sets.
 \end{proof}

 \begin{exone}
  Notice this does not generalize to other elementary toposes. Let $G = S_3$ the group of permutation of $3$ elements.
  The category of $G$-sets is an elementary topos as it is a category of presheaves.
  Let $S = \{ 1,2,3 \}$, which comes with an obvious $G$ action. 
  Then $S$ has no nontrivial automorphisms. Indeed, if $\sigma: S \to S$ is an automorphism
  Then for every element $\tau$ in $G$ we need to have $\sigma \tau = \tau \sigma$ in order to satisfy the equivariance condition.
  However, this is only satisfied by the identity as $S_3$ has a trivial center.
  \par 
  Thus $S$ has a no non-trivial automorphism in the category of $G$-sets. By Example \ref{Ex Uni over point} we deduce that the map 
  $S \to *$ is univalent. However, this map is not mono as a mono map in the category of presheaves is an injection of the underlying sets.
 \end{exone}

 \begin{remone}
  What we observed in this subsection is that an elementary topos is very well suited for the study of mono univalent maps, 
  but cannot understand univalent maps that are not mono. The failure stems from the fact that all objects in a category are $0$-truncated
  or, in other words, we only have $hom$-sets rather than mapping spaces and thus there is no hope of every classifying 
  non mono univalent maps. 
 \end{remone}
 
 The issue described here should motivate us to develop a higher category 
 that is able to classify all univalent maps.
 Following the analogy above such a higher category should be called an {\it elementary higher topos}.


\begin{thebibliography}{9}

%
 
 
 
 \bibitem[GK17]{GK17}
 D. Gepner, J. Kock. 
 {\it Univalence in locally cartesian closed $\infty$-categories}. 
 Forum Mathematicum. Vol. 29. No. 3. De Gruyter, 2017.
 
 \bibitem[KL12]{KL12}
 C. Kapulkin, P. LeFanu Lumsdaine. 
 {\it The simplicial model of univalent foundations (after Voevodsky)}, 
 arXiv preprint arXiv:1211.2851 (2012).
 
 \bibitem[Lu09]{Lu09}
 J. Lurie,
 {\it ($\infty$, $2$)-Categories and the Goodwillie Calculus I},
 arXiv preprint arXiv:0905.0462 (2009).
 
 \bibitem[MM92]{MM92}
 S. MacLane, I. Moerdijk,
 {\it Sheaves in geometry and logic: a first introduction to topos theory}, 
 (1992).
 
 \bibitem[Ra11]{Ra11}
 Alberto Garc{\'i}a-Raboso,
 {\it Stable $\infty$-Categories},
 \href{http://pages.uoregon.edu/njp/garcia.pdf}{http://pages.uoregon.edu/njp/garcia.pdf}
 
 \bibitem[Ra17a]{Ra17a}
 N. Rasekh,
 {\it Yoneda Lemma for Simplicial Spaces}.
 arXiv preprint arXiv:1711.03160 (2017).
 
 \bibitem[Ra17b]{Ra17b}
 N. Rasekh,
 {\it Cartesian Fibrations and Representability}.
 arXiv preprint arXiv:1711.03670 (2017).
 
 \bibitem[Ra18a]{Ra18a}
 N. Rasekh,
 {\it An Introduction to Complete Segal Spaces}.
 arXiv preprint arXiv:1805.03131 (2018).
 
 \bibitem[Re01]{Re01}
 C. Rezk, 
 {\it A model for the homotopy theory of homotopy theory}, 
 Trans. Amer. Math.Soc., 353(2001), no. 3, 973-1007.
 
 \bibitem[Re16]{Re16}
 C. Rezk,
 {\it Stuff about Quasicategories}.
 \href{https://faculty.math.illinois.edu/~rezk/595-fal16/quasicats.pdf}{https://faculty.math.illinois.edu/~rezk/595-fal16/quasicats.pdf}
 
 \bibitem[RS17]{RS17}
 E. Riehl, M. Shulman,
 {\it A type theory for synthetic $\infty$-categories},
 Higher Structures, Vol 1, No 1 (2017), 147-224.
 
 \bibitem[UF13]{UF13}
  The Univalent Foundations Program,
  {\it Homotopy Type Theory: Univalent Foundations of Mathematics},
  \href{https://homotopytypetheory.org/book}{https://homotopytypetheory.org/book},
  Institute for Advanced Study,
  2013
  
\end{thebibliography}
\end{document}